\documentclass[12pt]{article}

\usepackage{bm}
\usepackage{amsbsy}
\usepackage{amsmath}
\usepackage{amssymb}
\usepackage{amsthm}
\usepackage{authblk}
\usepackage{bbm}
\usepackage{color}
\usepackage{enumitem}
\usepackage{geometry}
\usepackage{graphicx}
\usepackage{natbib}
\usepackage[hidelinks]{hyperref}
\usepackage{mathtools}
\usepackage{setspace}
\usepackage{hyperref}

\theoremstyle{plain}
\newtheorem{theorem}{Theorem}

\newtheorem{corollary}{Corollary}

\newcommand{\RNum}[1]{\uppercase\expandafter{\romannumeral #1\relax}}

\def\FDP{\mathrm{FDP}}
\def\FDR{\mathrm{FDR}}

\def\Var{\mathrm{Var}}
\def\Cov{\mathrm{Cov}}

\def\bD{\mathbf{D}}
\def\bX{\mathbf{X}}
\def\bY{\mathbf{Y}}
\def\bT{\mathbf{T}}
\def\b1{\mathbf{1}}

\def\bmu{\boldsymbol{\mu}}
\def\bSigma{\boldsymbol{\Sigma}}

\def\calA{\mathcal{A}}
\def\calH{\mathcal{H}}

\title{Asymptotic Uncertainty of False Discovery Proportion for Dependent $t$-Tests}
\author{Meng Mei and Yuan Jiang\thanks{Yuan Jiang is the corresponding author. This research is supported in part by National Institutes of Health grant R01 GM126549.}}
\affil{Department of Statistics\\ Oregon State University}
\date{\today}

\begin{document}

\maketitle

\begin{abstract}

Multiple testing is a fundamental problem in high-dimensional statistical inference. Although many methods have been proposed to control false discoveries, it is still a challenging task when the tests are correlated to each other. To overcome this challenge, various methods have been proposed to estimate the false discovery rate (FDR) and/or the false discovery proportion (FDP) under arbitrary covariance among the test statistics. An interesting finding of these works is that the estimation of FDP and FDR under weak dependence is identical to that under independence. However, \citet{meng2021asymptotic} pointed out that unlike FDR, the asymptotic variance of FDP can still differ drastically from that under independence, and the difference depends on the covariance structure among the test statistics. In this paper, we further extend this result from $z$-tests to $t$-tests when the marginal variances are unknown and need to be estimated. With weakly dependent $t$-tests, we show that FDP still converges to a fixed quantity unrelated to the dependence structure, and further derive the asymptotic expansion and uncertainty of FDP leading to similar results as in \citet{meng2021asymptotic}. In addition, we develop an approximation method to efficiently evaluate the asymptotic variance of FDP for dependent $t$-tests. We examine how the asymptotic variance of FDP varies as well as the performance of its estimators under different dependence structures through simulations and a real-data study.

\end{abstract}

\section{Introduction}

Multiple hypothesis testing is a fundamental problem in high-dimensional statistical inference. The early-stage research focused on controlling familywise error rate (FWER) \citep{bonferroni1936teoria, vsidak1967rectangular, holm1979simple, simes1986improved, holland1987improved, hochberg1988sharper, rom1990sequentially} and generalized familywise error rate (gFWER) \citep{dudoit2004multiple, pollard2004choice, lehmann2012generalizations}. However, these two criteria only allow one or at most a pre-determined number of false discoveries among all hypothesis tests. As the number of simultaneous hypothesis tests increases, controlling FWER and gFWER become too conservative and lack in power. In \citet{benjamini1995controlling}, false discovery rate (FDR)  was first introduced for large-scale multiple hypothesis testing. FDR is defined as the expected value of false discovery proportion (FDP), while FDP is a random variable defined as the proportion of false discoveries among all discoveries, i.e.,
\[ \FDR = E(\FDP) = E\left( \frac{\# \text{rejections among true null hypotheses}}{\# \text{rejections}} \right), \]
and $\FDR = 0$ if $R = 0$. Since then, numerous methods have been proposed to control FDR, such as \citet{benjamini1995controlling} and \citet{storey2002direct}.

The above procedures have been shown to control FDR for independent test statistics successfully. Moreover, \citet{benjamini2001control} showed that the Benjamini-Hochberg procedure can also control FDR for test statistics under positive dependence, and \citet{storey2004strong} proved that method in \citet{storey2002direct} can also control FDR for test statistics under weak dependence. However, the dependence structure between test statistics can be arbitrary in practice. \citet{efron2007correlation} shows that the correlation structure between test statistics plays an important role in multiple hypothesis testing procedure. \citet{sun2009large} and \citet{clarke2009robustness} proposed multiple hypothesis testing methods under a more robust dependence structure.

As a milestone paper, \citet{fan2012estimating} incorporated the dependence information among the test statistics into the asymptotic evaluation of FDP. Their principal factor approximation (PFA) procedure detects the principal factors from the dependence and use these factors to form a new random variable in estimating FDP. In particular, when the dependence among the test statistics is weak, there is no factor detected by PFA and thus FDP can be approximately estimated by a fixed value that depends on the signal magnitudes and variances of each individual test statistic. In other words, the asymptotic limit of FDP is not affected by the correlation between the test statistics. Later, \citet{meng2021asymptotic} (Chapter 2 of this dissertation) pointed out that even under weak dependence structure, there is still a difference in the asymptotic variance of FDP between the weak dependence and independence.

%To take advantage of the variance of FDP in multiple testing procedure,  \citet{delattre2016empirical} pointed out that FDP will follow an asymptotic normal distribution under stationarity assumption, where the covariance matrix between the test statistics vanished if the distance becomes sufficiently large. 

Although the dependence information has been used in these methods, one constraint they shared is that the population covariances are assumed to be known. In practice, the population marginal variances are usually unknown and need to be estimated; the underlying distribution of the resultant test statistics become the $t$-distribution instead of the normal distribution. Additionally, the estimated covariance structure introduces additional error in FDP estimation. \citet{fan2017estimation} extended the results in \citet{fan2012estimating} to unknown dependence, with a few assumptions that the estimated covariance structure has to be precise enough in order to apply the PFA method.

In this article, we mainly focus on extending the methods and theory in \citet{meng2021asymptotic} (Chapter 2 of this dissertation) from $z$-tests to $t$-tests. We first introduce the theoretical formula of the asymptotic variance of FDP for weakly dependent $t$-tests that is fairly similar to that for $z$-tests. Then, we propose a practical method to estimate the asymptotic variance of FDP with unknown population dependence. Finally, we validate the established theory and illustrate the applicability of the proposed methods through both synthetic and real multiple hypothesis testing problems.

\section{Existing Results}
\label{Sec-4}

% Add notation to more general situation

Suppose that the observed data $\{\bX_i\}_{i = 1}^n$ are $p$-dimensional random vectors where $\bX_i = (X_{i1},\ldots,X_{ip})^T$ are i.i.d.\ samples from $\bX = (X_1,\ldots,X_p)^T \sim N_p(\bmu, \bSigma)$. For all the means $\mu_j$'s in the mean vector $\bmu = (\mu_1, \dots, \mu_p)^T$, a small group of them might be true signals and the task is to detect which locations in this high-dimensional vector are actually signals. This is equivalent to performing a multiple hypothesis testing problem for the null hypotheses $H_{0j}: \mu_j = 0$ versus the alternative hypotheses: $H_{0j}: \mu_j \ne 0$. Let $\calH_0 = \{j: \mu_j = 0\}$ and $\calH_1 = \{j: \mu_j \ne 0\}$ be the sets of indices for the true null hypotheses and the true alternative hypotheses, respectively; correspondingly, let $p_0$ and $p_1$ denote respectively their cardinalities.

If the variances of $X_1,\ldots,X_p$ are known as $\sigma_1^2,\ldots,\sigma_p^2$, consider the test statistics $Z_j = \sqrt{n}\bar{X}_j/\sigma_{j}$, where $\bar{X}_j = n^{-1}\sum_{i=1}^n X_{ij}$. Define $t_j = \mathrm{1}(|Z_j| > |z_{t/2}|)$ to be the $j$th test, where $t \in [0, 1]$ is a fixed threshold, and $z_{t/2}$ is the $t/2$ quantile of a standard normal distribution. Define $R(t) = \sum_{j=1}^p t_j$ to be the total number of discoveries among these $p$ hypothesis tests, while $V(t)=\sum_{j \in \mathcal{H}_0} t_j$ and $S(t)=\sum_{j \in \mathcal{H}_1} t_j$ correspond to the number of true discoveries and false discoveries in these tests. Then, the false discovery rate (FDR) is defined as $\FDR(t) = E[\FDP(t)] = E[V(t)/R(t)]$, the expected value of the false discovery proportion (FDP).

\citet{fan2012estimating} proposed a method in estimating FDP under arbitrary covariance structure among the test statistics. One immediate result from \citet{fan2012estimating} is that, under weak dependence among the test statistics, which is defined as:
\begin{equation}\label{weak.dependence}
p^{-2}\sum_{j,k}|\sigma_{jk}| = O\left(p^{-\delta}\right) \text{ for some } \delta>0,
\end{equation} 
FDP converges to a deterministic quantity. Specifically, 
\begin{equation*} 
\lim_{p\to\infty} \left[ \mathrm{FDP}(t) - \frac{p_0 t}{\sum_{j=1}^{p}\{\Phi(z_{t/2}+\mu_{j}) + \Phi(z_{t/2}-\mu_{j})\}} \right] = 0, \ \text{a.s.},
\end{equation*}
where $\Phi(\cdot)$ denote the cumulative distribution function of the standard normal distribution. The asymptotic mean of FDP has the same form as the above asymptotic limit. Thus, the asymptotic mean of FDP depends only on the population mean $\bmu$ and the threshold $t$, and it would be identical for simultaneous hypothesis tests with the same signals but different dependence structures as long as the dependence is weak. 

By contrast, \citet{meng2021asymptotic} pointed out that the uncertainty of FDP could still vary with weak dependence structures. In detail, \citet{meng2021asymptotic} derived an asymptotic expansion of FDP as the sum of three terms: an asymptotic mean that is a constant, a stochastic term as a linear combination of the tests $t_j, j=1,\ldots,p$, and an asymptotically negligible remainder term. The asymptotic variance of FDP is determined by the stochastic term and thus varies with different dependence structures among the test statistics. Specifically, the asymptotic variance of FDP composes two parts, one of which includes the variances of individual tests and the other represents the pair-wise covariances between the tests.

Both \citet{fan2012estimating} and \citet{meng2021asymptotic} focused on multivariate normal test statistics, in which the marginal variances, $\sigma_1^2, \ldots, \sigma_p^2$, are assumed to be known. However, the marginal variances are often unknown in real applications and need to be estimated. When the marginal variances are estimated, $t$-tests are used instead of $z$-tests. \citet{fan2017estimation} showed that the asymptotic result in \citet{fan2012estimating} would still hold under certain conditions for dependent $t$-tests. In parallel, we propose to further explore the asymptotic uncertainty of FDP for dependent $t$-tests.

\section{Asymptotic Uncertainty of FDP for Dependent $t$-Tests}
\label{t-Sec-5}

\subsection{Theoretical results}
\label{t-Sec-5.1}

Under the framework of multiple testing in Section \ref{Sec-4}, when the variances $\sigma_1^2, \ldots, \sigma_p^2$ are unknown, the sample variances $\hat{\sigma}_1^2, \ldots, \hat{\sigma}_p^2$ are often used instead, where $\hat{\sigma}_j^2 = (n-1)^{-1}\sum_{i=1}^{n}(X_{ij} - \bar{X}_j)^2$. To perform the same multiple testing problem, the $t$-test statistics $\bT = (T_1, \dots, T_p)^T = \sqrt{n}\hat{\bD}^{-1}\bar{\bX}$ are considered, where $\bar{\bX}=n^{-1}\sum_{i=1}^n \bX_i$, and $\hat{\bD} = \text{diag}(\hat{\sigma}_1, \dots, \hat{\sigma}_p)$ is a diagonal matrix with the diagonal elements being the sample standard deviations.

The distribution of $\bT$ is no longer a multivariate normal distribution. For the true null hypotheses, each $T_j$ follows the $t_{n-1}$-distribution and the hypothesis tests become $t$-tests. Similar to Section \ref{Sec-4}, define the $j$th test to be $t_j = \mathrm{1}(|T_j| > |q_{t/2}|)$, where $q_{t/2}$ is the $t/2$ quantile of the $t_{n-1}$-distribution. Denote $\xi_j = P(|T_j|>|q_{t/2}|) = E(t_j)$, and $\bar\xi = \frac{1}{p_1}\sum_{j\in \mathcal{H}_1} \xi_j$. Consequently, $V(t)$, $S(t)$, $R(t)$, and $\FDP(t)$ are defined in the same way as in Section \ref{Sec-4} but based on the $t$-tests instead of the $z$-tests.

From \citet{fan2012estimating}, FDP from $z$-tests converges to a deterministic quantity when the dependence among the test statistics is weak as defined in (\ref{weak.dependence}). As follows, we extend this result from $z$-tests to $t$-tests. 

\begin{theorem} \label{Thm-4}
Suppose $(X_{1},\ldots,X_{p})^{T} \sim N((\mu_{1},\ldots,\mu_{p})^{T},\bSigma)$ with unit variances. Assume that $(X_{1},\dots,X_{p})^T$ are weakly dependent as defined in (\ref{weak.dependence}) and that $\limsup_{p\to \infty} p_0t/(p_1\bar \xi) < 1$. Then,
\begin{equation} \label{asymptotic.limit.t}
\lim_{p\to\infty} \left[\mathrm{FDP}(t) -\frac{p_0 t}{\sum_{j=1}^{p} E\{\Phi(q_{t/2}\hat{\sigma}_j + \mu_j) + \Phi(q_{t/2}\hat{\sigma}_j -\mu_j)\}}\right] = 0, \ \text{a.s.}
\end{equation}
\end{theorem}

It is apparent that the asymptotic limit (or equivalently, asymptotic mean) of FDP does not depend on the covariance matrix $\bSigma$. However, we will show that the covariance matrix $\bSigma$ plays a key role in the asymptotic uncertainty of FDP. To evaluate how $\bSigma$ affects the asymptotic uncertainty of FDP, it is critical to consider the dependence among the $t$-tests $t_1,\ldots,t_p$. To this end, we will present the joint distribution of $\bT$ briefly as follows.

Each pair of $(T_i, T_j)$ follows the so-called ``bivariate $t$-distribution'' \citep{siddiqui1967bivariate}. However, the joint distribution of $\bT$ is not the conventional multivariate $t$-distribution introduced in \citet{kotz2004multivariate}. Instead, the joint distribution of $\bT$ is referred to as the dependent $t$-distribution to avoid any confusion \citep{fan2017estimation}. A nice property of the dependent $t$-distribution is that it relates closely with the multivariate normal distribution. As $\bX_1,\ldots,\bX_n$ are i.i.d.\ samples from $N_p(\bmu, \bSigma)$, $\bar{\bX}$ and $\hat{\bSigma}$ are independent. Therefore, conditioning on $\hat{\sigma}_1^2, \ldots, \hat{\sigma}_p^2$, $\bT$ follows a multivariate normal distribution as follows:
\begin{equation} \label{hierarchical.represetation}
\bT | \hat{\sigma}_1^2, \ldots, \hat{\sigma}_p^2 \sim N(\sqrt{n}\hat{\bD}^{-1}\bmu, \hat{\bD}^{-1}\bSigma\hat{\bD}^{-1}). 
\end{equation}

This above hierarchical representation of the dependent $t$-distribution enables us to generalize the result on the asymptotic uncertainty of FDP from $z$-tests to $t$-tests. First, we generalize Theorem 1 in \citet{meng2021asymptotic} to dependent $t$-tests.

\begin{theorem} \label{Thm-5} 
Assume that $(X_{1},\dots,X_{p})^T \sim N((\mu_{1},\ldots,\mu_{p})^{T},\bSigma)$ with unit variances and are weakly dependent as defined in (\ref{weak.dependence}), that $\limsup_{p\to \infty} p_0t/(p_1\bar \xi) < 1$, and that for a universal constant $C>0$,
\begin{align}
&\sum_{j\neq k; j,k\in \mathcal{H}_0} \sigma_{jk}^2 \geq \frac{C_t^{\max}}{E^2\{\phi(\hat{\sigma}_j z_{t/2})|\hat{\sigma}_j z_{t/2}|\}} \sum_{j\in \mathcal{H}_1, k\in \mathcal{H}_0, \mu_j\in[-\mu_t, \mu_t]} \sigma_{jk}^2, \label{Cond-4} \\
&\sum_{j\neq k; j,k\in \mathcal{H}_0} \sigma_{jk}^2 + p \geq C \sum_{j\neq k; j,k\in \mathcal{H}_1} \sigma_{jk}^2, \label{Cond-5}\\
&\sum_{j\neq k} \sigma_{jk}^4 = o\left(\sum_{j\neq k; j,k\in \mathcal{H}_0} \sigma_{jk}^2  + p_0\right). \label{Cond-6}
\end{align}
where $C_t^{\max} = \sup_{\mu \in (-\mu_t, \mu_t)} H(\mu)$, and $\mu_t$ is the unique root of 
$$H(\mu) = E(\phi(q_{t/2}\hat{\sigma}_j)\hat{\sigma}_j \left\{ \phi(|q_{t/2}|\hat{\sigma} + \mu)(|q_{t/2}|\hat{\sigma} + \mu) + \phi(|q_{t/2}|\hat{\sigma} - \mu)(|q_{t/2}|\hat{\sigma} - \mu)\right\}) \text{ for } \mu\in (0, \infty).$$
Then, we have the following asymptotic expansion of $\FDP(t)$:
\begin{equation} \label{asymptotic.expansion.unknown}
\FDP(t) = \frac{E(\bar V)}{E(\bar R)} + m(\bar{V},\bar{R}) + r(\bar V, \bar R), 
\end{equation}
where $\bar V = V(t)/p$, $\bar R = R(t)/p$,
\begin{equation*}
m(\bar{V},\bar{R}) =  \frac{\bar{V}}{E(\bar{R})} - \frac{E(\bar{V})}{\{E(\bar{R})\}^2} \bar{R},
\end{equation*}
and the remainder term $r(\bar V, \bar R)$ satisfies that $E\{r^2(\bar{V}, \bar{R})\} = o[\Var\{m(\bar{V},\bar{R})\}]$.
\end{theorem}

The above theorem is very similar to Theorem 1 in \citet{meng2021asymptotic}. Conditions (\ref{Cond-4})--(\ref{Cond-6}) are almost identical to Conditions (7)--(9) in \citet{meng2021asymptotic}. The only difference is that the denominator in (\ref{Cond-4}) is different from the one in Condition (7) in \citet{meng2021asymptotic} due to the difference of the threshold value and the variances between the $z$-tests and $t$-tests. These conditions are mild under the weak dependence assumption in (\ref{weak.dependence}). See Remark 2 in \citet{meng2021asymptotic} for a detailed discussion of these conditions.

Theorem \ref{Thm-5} shows that $\FDP(t)$ can be decomposed into three terms: an asymptotic mean term $E(\bar V)/E(\bar R)$, an stochastic term $m(\bar{V},\bar{R})$ that is a linear combination of the tests $t_1,\ldots,t_p$, and an asymptotically negligible remainder term $r(\bar{V},\bar{R})$. The asymptotic variance of FDP is solely determined by the stochastic term $m(\bar{V},\bar{R})$ and can be easily expressed as the covariances between the tests $t_1,\ldots,t_p$. This leads to the asymptotic variance of $\FDP(t)$ as in the following corollary. 

\begin{corollary} 
\label{Cor-1}
With all the conditions in Theorem \ref{Thm-5} effective, we have
\begin{equation} \label{asymptotic.variance.unknown}
\lim_{p\to\infty}\dfrac{\Var\left\{\FDP(t)\right\}}{V_1(t) + V_2(t)} = 1, 
\end{equation}
where 
\begin{align}
V_1(t) ={} & \dfrac{p_1^2 \bar{\xi}^2}{(p_0 t + p_1\bar{\xi})^4} \times p_0t(1-t) + \dfrac{p_0^2 t^2}{(p_0 t + p_1\bar{\xi})^4} \times p_1\bar{\xi}(1 - \bar{\xi}), \label{V.1.unknown}\\
V_2(t) ={} & \dfrac{2p_1^2 \bar{\xi}^2}{(p_0 t + p_1\bar{\xi})^4} \sum_{\substack{j<k \\ j,k\in\calH_0}} \Cov(t_j,t_k)  - \dfrac{2p_0p_1t\bar{\xi}}{(p_0 t + p_1\bar{\xi})^4} \sum_{\substack{j \in \calH_0 \\ k \in \calH_1}} \Cov(t_j, t_k) \notag \\
&+ \dfrac{2p_0^2 t^2}{(p_0 t + p_1\bar{\xi})^4} \sum_{\substack{j<k \\ j,k\in\calH_1}} \Cov(t_j,t_k). \label{V.2.unknown}   
\end{align}
\end{corollary}

Similar to Corollary 1 in \citet{meng2021asymptotic}, $V_2(t)$ becomes $0$ if the test statistics are independent with each other. Thus, it represents the ``additional'' variance of FDP introduced by the dependence, while $V_1(t)$ represents the variance of FDP when all test statistics are independent with each other. 

Although Corollary \ref{Cor-1} provides an explicit formula for the asymptotic variance of FDP that depends on the covariances between the tests $t_1,\ldots,t_p$. However, since there is no explicit form for the cumulative distribution function for bivariate $t$-distribution, it would be challenging to get a precise result for the joint probabilities $P(|T_j| > |q_{t/2}|, |T_k| > |q_{t/2}|)$. A naive approach is to use the Monte Carlo method by simulating pairs of bivariate $t$-distribution. However, this naive approach may be computationally intensive as it may require a large number of simulations especially when $t$ is small. In the next section, we propose a much more efficient method with similar precision to the naive Monte Carlo method.

\subsection{Evaluating the asymptotic variance of FDP}
\label{t-Sec-5.2} %Change title

In this subsection, we propose an approximation method to compute the asymptotic variance of FDP to reduce the computational burden while maintaining the estimation accuracy compared to the naive Monte Carlo method. 

Recall the conditional distribution of $\bT$ given $\hat{\sigma}_1^2, \ldots, \hat{\sigma}_p^2$ in (\ref{hierarchical.represetation}). Denote the joint probability density function of $(\hat{\sigma}_j^2, \hat{\sigma}_k^2)$ by $f(\hat{\sigma}_j^2, \hat{\sigma}_k^2)$, then the covariance of $(t_j, t_k)$ can be computed with a bivariate normal distribution as:
\begin{equation} \label{double.integral}
\Cov(t_j, t_k) = \int_{0}^{\infty}\int_{0}^{\infty} \Cov(t_j, t_k|\hat{\sigma}_j^2, \hat{\sigma}_k^2) f(\hat{\sigma}_j^2, \hat{\sigma}_k^2) d\hat{\sigma}_j^2 d\hat{\sigma}_k^2,
\end{equation}
where $\Cov(t_j, t_k|\hat{\sigma}_j^2, \hat{\sigma}_k^2)$ equals the covariance of $\mathrm{1}(|T_j| > |q_{t/2}|)$ and $\mathrm{1}(|T_k| > |q_{t/2}|)$ where $(T_j, T_k) \sim N(\sqrt{n}\hat{\bD}_{\calA,\calA}^{-1}\bmu_{\calA}, \hat{\bD}_{\calA,\calA}^{-1}\bSigma_{\calA,\calA}\hat{\bD}_{\calA,\calA}^{-1})$ with $\calA = \{j,k\}$. This quantity is thus easily computable using the cumulative distribution function of a bivariate normal distribution.

To evaluate the double integral in (\ref{double.integral}), we need to figure out the joint distribution of $(\hat{\sigma}_j^2, \hat{\sigma}_k^2)$. Although it is known that $\hat{\sigma}_j^2 \sim \chi_{n-1}^2/(n-1)$ marginally, the joint distribution of $(\hat{\sigma}_j^2, \hat{\sigma}_k^2)$ is unknown. Fortunately, such a joint distribution can be approximated by a bivariate normal distribution when $n$ is large:
\begin{equation} \label{asymptotic.normality}
\sqrt{n-1} (\hat{\sigma}_j^2, \hat{\sigma}_k^2)^T \stackrel{d}{\longrightarrow} N\left[(\sigma_{j}^2, \sigma_{k}^2)^T, \left(\begin{matrix}
2\sigma_{j}^4 & 2\sigma_{jk}^2\\
2\sigma_{jk}^2 & 2\sigma_{k}^4
\end{matrix}\right)\right].
\end{equation}

Therefore, $\Cov(t_j, t_k)$ can be approximately evaluated by the double integral in (\ref{double.integral}) in which $f(\hat{\sigma}_j^2, \hat{\sigma}_k^2)$ is replaced by the asymptotic distribution in (\ref{asymptotic.normality}). Compared to the naive Monte Carlo method, this approximation approach significantly alleviates the computational burden and leads to a comparably accurate result. See the simulation results in Section \ref{t-Sec-6} for more details.

In practice, the mean $\bmu$ and the covariance matrix $\bSigma$ are often unknown and need to be estimated as they are parameters in (\ref{double.integral}). To estimate $\bmu$, we follow the two-step estimation procedure in \citet{meng2021asymptotic}. First, we use existing methods, such as the Langaas' method \citep{langaas2005estimating} and SLIM \citep{wang2010slim}, to estimate $\pi_0 = p_0/p$; second, the $\mu$'s corresponding to the $p_1 = p(1 - \pi_0)$ largest absolute values of the test statistics will be estimated as the test statistic values and the rest of the means will be set to $0$.

To estimate $\bSigma$, one naive approach is to use the sample covariance matrix $\hat\bSigma$. However, in a high-dimensional multiple testing problem, $n$ is usually much smaller than $p$. In such a case, the dependence in the low-rank sample covariance matrix $\hat\bSigma$ tends to be much larger than that in $\bSigma$. The dependence is highly related to the sum of squares of all the eigenvalues from the correlation matrix, which tends to be larger in a low-rank matrix. To resolve this issue, we will employ the regularized estimators of the covariance matrix to improve estimation accuracy while avoiding singularity. For example, \citet{peeters2020rags2ridges} proposed a ridge estimator of the inverse covariance matrix that is full-rank and reduces the additional dependence caused by the low-rank estimation. We will use the inverse of such an estimator and still call this method the ridge estimation for simplicity. In the simulations, we will show that the ridge estimation method yields a more accurate estimator for the covariance matrix than the sample covariance matrix and further improves the estimation of the asymptotic variance of FDP.    

\section{Simulation Studies}
\label{t-Sec-6}

We conduct two simulation studies. The first simulation study is to validate the theoretical result in the asymptotic variance of FDP and the second simulation study is to compare different approaches in estimating the asymptotic variance of FDP.

\subsection{Validate the asymptotic mean and variance of FDP}
\label{t-Sec-6.1}

We generate $(X_{i,1}, \dots, X_{i,p})^T \sim N((\mu_1, \dots, \mu_p), \bSigma)$, where $p$ is $500$ and $i = 1, \dots, 200$. The $p_1$ true alternative hypotheses $\calH_1$ are randomly located in all hypotheses and $\mu_i = 2|q_{t/2}|$ for $i \in \calH_1$. 

For the generation of $\bSigma$, we follow the same setting as in \citet{meng2021asymptotic}. First, a random sample of a $p$-dimensional vector $(Z_1, \dots, Z_p)^T$ with sample size 400 were generated and the sample covariance matrix $\bSigma^{\text{initial}}$ was calculated. Then, the PFA method from \citet{fan2012estimating} was used to remove the major principal factors from $\bSigma^{\text{initial}}$ to ensure weak dependence. The $p$-dimensional vectors $(Z_1, \dots, Z_p)^T$ were generated from one of the following six models:
\begin{itemize}
\item \textbf{[Equal correlation]} Let $(Z_{1},\dots,Z_{p})^{T} \sim N_{p}(0, \mathbf\Lambda)$, where $\mathbf\Lambda$ has diagonal element $1$ and off-diagonal element $1/2$.
\item \textbf{[Fan \& Song's model]} Let $\{Z_{k}\}_{k=1}^{1900}$ be iid $N(0,1)$ and $Z_{k} = \sum_{1=1}^{10}Z_{l}(-1)^{l+1}/5 + \sqrt{1-\dfrac{10}{25}}\varepsilon_{k},\ k=1901,\dots,2000,$
where $\{\varepsilon_{k}\}_{k=1901}^{2000}$ are standard normally distributed.
\item \textbf{[Independent Cauchy]} Let $\{Z_{k}\}_{k=1}^{2000}$ be iid Cauchy random variables with location parameter $0$ and scale parameter $1$.
\item \textbf{[Three factor model]} Let $Z_{j} = \rho_{j}^{(1)}W^{(1)}+\rho_{j}^{(2)}W^{(2)}+\rho_{j}^{(3)}W^{(3)}+H_{j}$, where $W^{(1)} \sim N(-2,1)$, $W^{(2)} \sim N(1,1)$, $W^{(3)} \sim N(4,1)$, $\rho_{j}^{(1)}$, $\rho_{j}^{(2)}$, $\rho_{j}^{(3)}$ are iid $U(-1,1)$, and $H_{j}$ are iid $N(0,1)$.
\item \textbf{[Two factor model]} Let $Z_{j} = \rho_{j}^{(1)}W^{(1)}+\rho_{j}^{(2)}W^{(2)}+H_{j}$, where $W^{(1)}$ and $W^{(2)}$ are iid $N(0,1)$, $\rho_{j}^{(1)}$ and $\rho_{j}^{(2)}$ are iid $U(-1,1)$, and $H_{j}$ are iid $N(0,1)$.
\item \textbf{[Nonlinear factor model]} Let $Z_{j} = \mathrm{sin}\left(\rho_{j}^{(1)}W^{(1)}\right)+\mathrm{sign}\left(\rho_{j}^{(2)}\right)\exp\left(|\rho_{j}^{(2)}|W^{(2)}\right)+H_{j}$, where $W^{(1)}$ and $W^{(2)}$ are iid $N(0,1)$, $\rho_{j}^{(1)}$ and $\rho_{j}^{(2)}$ are iid $U(-1,1)$, and $H_{j}$ are iid $N(0,1)$.
\end{itemize}

In this simulation study, we consider three choices for the number of alternative hypotheses $p_1 = 10, 20, 50$, and two choices for the thresholds $t = 0.02, 0.05$. For each of above six simulation settings, we compare the empirical mean of FDP from 2000 replicates and the asymptotic mean of FDP based on (\ref{asymptotic.limit.t}). In addition, we compare the variances of FDP evaluated via three methods. First, we calculate the empirical variance of FDP based on 1000 replicates (labeled as Empr). Second, we calculate the asymptotic variance of FDP based on (\ref{asymptotic.variance.unknown}) in which $\Cov(t_i,t_j)$ are evaluated based on both the naive Monte Carlo method with 10000 replicates (labeled as Asym-MC) and the approximation method in Section \ref{t-Sec-5.2} (labeled as Asym-AP).

The results are summarized in Tables \ref{Table-6} and \ref{Table-7}. From Table \ref{Table-6}, the asymptotic mean of FDP is slightly higher but still quite close to the empirical means of FDP from all six models. This validates our conclusion in Theorem \ref{Thm-4}. From Table \ref{Table-7}, the Monte Carlo-based asymptotic standard deviation of FDP is very close to the empirical standard deviation, which validates our result in Theorem \ref{Thm-5}. In addition, the two asymptotic standard deviations of FDP based on Monte Carlo and the approximation method are close to each other. This shows the validity of the approximation method in Section \ref{t-Sec-5.2}. In terms of computational complexity, the approximation method has a constant computational time as the significant level $t$ decreases but Monte Carlo would have a quadratic increase in computational time. In our simulation study, the approximation method is about 3 times faster than Monte Carlo when $t = 0.02$; it could be about 50 times faster if $t$ was set as $0.005$.

\begin{table}
\centering
 \caption{Asymptotic and empirical means of FDP ($\times100$). Asym: asymptotic mean of FDP as in (\ref{asymptotic.limit.t}); Equal/Fan/Cauchy/2f/3f/Nonlinear: empirical mean of FDP from 1000 replicates based on each model.}
\begin{tabular}{c c c  c  c  c  c  c  c c}\\
\hline
$p_1$ & $t$ & Asym & Equal & Fan & Cauchy & 2f & 3f & Nonlinear\\ \hline
10 & 0.02 & 51.0 & 47.8 & 48.4 & 47.7 &48.1 & 47.8 & 48.3 \\
& 0.05 & 72.1 & 70.5 & 70.9 & 70.9 &70.5 & 70.4 & 70.6 \\ \hline
 20 & 0.02 & 33.7 & 31.7 & 32.0 &31.7 &31.7 & 32.0 & 31.5 \\ 
  &0.05 & 55.9 & 54.3 & 54.5 & 54.2 &54.6 & 54.3 & 54.2 \\ \hline
 50 & 0.02 & 16.0 &15.3 & 15.0 &15.0 &15.1 & 15.0 & 15.0 \\ 
 &0.05 & 32.2 &30.9 & 31.3 &31.3 &31.1 & 31.2 & 30.9 \\ \hline
 \end{tabular}
 \label{Table-6}
 \end{table}
 
\begin{table}
\centering
\caption{Asymptotic and empirical standard deviation of FDP ($\times 100$). Empr: empirical standard deviation of FDP from 1000 replicates; Asym-MC: Asymptotic standard deviation of FDP based on Monte Carlo; Asym-AP: Asymptotic standard deviation of FDP based on the approximation method in Section \ref{t-Sec-5.2}.}
\begin{tabular}{c  c c  c  c  c  c  c  c}\\
\hline
$p_1$ &$t$ & Method &  Equal & Fan & Cauchy & 2f & 3f & Nonlinear\\ \hline
10 &0.02  & Empr &9.99 & 9.17 &10.58 & 10.18 & 10.26 & 10.15  \\
  & & Asym-MC &9.20 & 8.30 &9.52 & 9.27 & 9.22 & 9.24 \\
 & & Asym-AP &9.45 & 8.47 &9.71 & 9.46 & 9.47 & 9.48 \\ \hline
 &0.05 & Empr &5.79 &4.91 &5.72 & 5.93 &5.79 & 5.54  \\
  & & Asym-MC & 5.16 & 4.50 &5.23 & 5.15 & 5.17 & 5.19\\ 
 & & Asym-AP & 5.21 & 4.54 &5.27 & 5.21 & 5.21 & 5.22\\ \hline
20 & 0.02& Empr &8.52 & 7.75 &8.81 & 8.55 &8.39 & 8.66  \\ 
 && Asym-MC & 8.30 & 7.41 &8.53 & 8.28 & 8.30 & 8.25\\ 
 && Asym-AP & 8.50 & 7.63 &8.73 & 8.51 & 8.51 &8.53\\ \hline
 &0.05 & Empr &6.50 &5.93 &6.83 & 6.63 &6.71 & 6.78 \\ 
 && Asym-MC &6.28 & 5.50 &6.33 & 6.27 & 6.26 & 6.29\\
 && Asym-AP &6.36 & 5.54 &6.43 & 6.36 & 6.37 & 6.38\\ \hline
 50&0.02 & Empr &5.11 &4.59 &5.41 & 4.92 &5.16 & 4.96 \\ 
 && Asym-MC &5.11 & 4.60 &5.24 & 5.09 &5.10 & 5.10\\ 
 && Asym-AP &5.23 & 4.71 &5.38 & 5.23 &5.23 & 5.24\\ \hline
 & 0.05 & Empr &5.66 &5.08 &5.84 & 5.59 &5.76 & 5.85 \\
 && Asym-MC &5.61 &4.91 &5.72 & 5.62 & 5.62 & 5.62\\ 
 && Asym-AP &5.70 &4.99 &5.79 & 5.71 &5.71 & 5.72\\ \hline
 \end{tabular}
 \label{Table-7}
 \end{table}

% Simulation result

\subsection{Estimate the asymptotic variance of FDP}
\label{t-Sec-6.2}

In this simulation study, the data were generated with the same procedure as in the simulation study in Section \ref{t-Sec-6.1} but only under the Independent Cauchy model. The simulation results are indeed similar under the other models and we omit them due to the limited space.

To estimate the asymptotic variance of FDP, we follow the approaches presented in Section \ref{t-Sec-5.2} to estimate $\bmu$ and $\bSigma$. As there are two methods to estimate $\bmu$ (Langgas and SLIM) and two methods to estimate $\bSigma$ (sample covariance matrix and the ridge estimator), we apply a total of four methods to estimate the asymptotic variance of FDP, labeled respectively as L-S, L-R, S-S, S-R. %To evaluate these estimation methods, we first evaluate the performance for these methods in their own aspect. Since The performance for Langaas and SLIM has been evaluated in \citet{meng2021asymptotic}, we will focus on evaluate two methods in estimating $\bSigma$. 
For comparison, we also present the empirical variance of FDP as well as the asymptotic variance of FDP with the true population parameters $\bmu$ and $\bSigma$.

\begin{table}
    \centering
    \caption{Mean ($\times 100$) and standard deviation ($\times 10^4$, in parenthesis) of the mean absolute error in estimating $\bSigma$ from 100 replicates. MAE(S): mean absolute error for sample covariance matrix; MAE(R): mean absolute error for the ridge estimator in \citet{peeters2020rags2ridges}.}
    \begin{tabular}{cccccc} \hline
         $p_1$ & $t$ & MAE(S) & MAE(R) \\ \hline
         10 & 0.02 & 5.64 (2.56) & 4.03 (3.89) \\ 
         & 0.05 & 5.64 (2.75) & 4.03 (4.19)\\ \hline
         20 & 0.02 & 5.64 (2.54) & 4.02 (3.86)\\
         & 0.05 & 5.64 (2.61) & 4.03 (3.97)  \\ \hline
         50 & 0.02 & 5.64 (2.43) & 4.01 (3.70) \\ 
         & 0.05 & 5.64 (2.67) & 4.01 (4.24) \\ \hline
    \end{tabular}
    \label{Table-10}
    
\end{table}

\begin{table}
\centering
 \caption{Mean ($\times 100$) and standard deviation ($\times 100$, in parenthesis) of the estimated asymptotic standard deviation of FDP from 100 replicates. Empr: empirical standard deviation of FDP from 1000 replicates; Asym: asymptotic standard deviation of FDP with true $\bmu$ and $\bSigma$; L-S: Langgas \& sample covariance matrix; L-R: Langgas \& ridge estimator; S-S: SLIM \& sample covariance matrix; S-R: SLIM \& ridge estimator.}
\begin{tabular}{c c c  c  c  c  c  c  c c}\\
\hline

$p_1$ & $t$ & Empr &  Asym & L-S & L-R & S-S & S-R\\ \hline
10 & 0.02 & 10.58 & 9.71 & 11.43 (1.42) & 9.30 (1.04) & 12.05 (0.80) & 9.60 (0.95)  \\ 
& 0.05 & 5.72 & 5.27 & 7.95 (1.10) & 6.31 (0.88) & 7.49 (1.49)  & 5.96 (1.13) \\ \hline
20 & 0.02 & 8.81 & 8.73 & 10.17 (1.30) & 8.28 (0.99) & 11.09 (0.82) & 8.79 (0.63)  \\ 
& 0.05 & 6.83 & 6.43 & 8.54 (0.43) & 6.70 (0.30) & 8.46 (0.59) & 6.69 (0.45)\\  \hline
50 & 0.02 & 5.41 & 5.38 & 6.33 (0.77) & 5.14 (0.55) & 6.89 (0.57) & 5.48 (0.43) \\ 
& 0.05 & 5.84 & 5.79 & 7.44 (0.56) & 5.73 (0.49) & 7.75 (0.32) & 6.04 (0.26) \\ \hline
 \end{tabular}
 \label{Table-8}

 \end{table}

Table \ref{Table-10} shows the simulation results in comparing the two different methods for estimating $\bSigma$. It is immediately seen that the ridge estimator outperforms the sample covariance matrix in terms of mean absolute error in all settings. Table \ref{Table-8} presents the simulation results in the estimation of the asymptotic standard deviation of FDP. It is immediately seen that the empirical standard deviation is close to the asymptotic standard deviation with the true population parameters, which again verifies the theoretical result in Theorem \ref{Thm-5}. To evaluate the estimation methods, on the one hand, comparing the two methods in estimating $\bmu$, both Langgas and SLIM lead to similarly accurate estimates of the asymptotic standard deviation. This agrees with the observations in \citet{meng2021asymptotic} that both Langgas and SLIM produce reasonably precise estimate for the proportion of true null hypotheses and thus precise estimate for the FDP uncertainty. On the other hand, comparing the two methods in estimating $\bSigma$, the ridge estimator is clearly better than the sample covariance matrix, with a much closer estimated value to both the empirical standard deviation and the asymptotic standard deviation with true population parameters.

\section{Real Data}

We applied our proposed method to a genome-wide gene expression study \citep{noble2008regional} to show its applicability to real data. This study compared the intestinal gene expression level in patients with ulcerative colitis and in controls. Ulcerative colitis is a chronic relapsing inflammatory disease of the gastrointestinal tract, which is known as a complex clinical disease and occurs in genetically susceptible individuals. Researchers investigated 67 patients with ulcerative colitis and 31 controls from Western General Hospital, Edinburgh, UK. For each individual, paired endoscopic biospies were taken from 5 specific anatomical locations, and total RNA was extracted from each biopsy using the micro total RNA isolation from animal tissues protocol. As a result, 129 paired biopsies were taken from patients with ulcerative colitis, and 73 paired biopsies were taken from controls. The gene expression data are available from \url{https://www.ncbi.nlm.nih.gov/geo/query/acc.cgi?acc=GSE11223}. To reduce the computational burden, we only analyze the expression data of the genes on chromosome 10.

Let $\mathbb{X} = (\bX_1,\ldots,\bX_n)$ denote the data for the $p = 1205$ genes' expression levels for the ulcerative colitis group ($n = 129$), and $\mathbb{Y} = (\bY_1,\ldots,\bY_m)$ for the control group ($m = 73$). Assume that all the samples are independent and that the gene expressions from each sample follow a multivariate normal distribution with group-specific means but an identical covariance matrix. In other words, $\bX_i \stackrel{\text{iid}}{\sim} N(\bmu_X, \bSigma)$ for $i = 1, \ldots, n$ and $\bY_i \stackrel{\text{iid}}{\sim} N(\bmu_Y, \bSigma)$ for $i = 1, \ldots, m$. Thus, multiple hypothesis testing can be applied to identify the differentially expressed genes as:
\begin{equation} \label{hypotheses.3}
H_{0j}: \mu_{X,j}=\mu_{Y,j} \quad \text{versus} \quad H_{1j}:\mu_{X,j} \neq \mu_{Y,j},\quad j=1,\dots,p.
\end{equation}
For two-sample $t$-test, the test statistic $T = \sqrt{nm}(\Bar{\bX} - \Bar{\bY})/\sqrt{n+m}$ can be used, where $\Bar{\bX}$ and $\Bar{\bY}$ are the sample means in the ulcerative coliti and control groups. Then, we have that $T \sim N(\bmu, \bSigma)$, where $\bmu = \sqrt{nm}(\bmu_X - \bmu_y)/\sqrt{n+m}$. Thus, testing the hypotheses in (\ref{hypotheses.3}) is equivalent to testing
\begin{equation} \label{hypotheses.4}
H_{0j}: \mu_{j}=0 \quad \text{versus} \quad H_{1j}:\mu_{j} \neq 0,\quad j=1,\dots,p.
\end{equation}

Since the population covariance matrices in the two groups are assumed to be identical, then the following pooled sample covariance matrix can be used:
\[\hat{\bSigma} = \dfrac{1}{n + m - 2}\left(\mathbb{X} - \Bar{\bX}\b1_n^T, \mathbb{Y} - \Bar{\bY}\b1_m^T\right)\left(\mathbb{X} - \Bar{\bX}\b1_n^T, \mathbb{Y} - \Bar{\bY}\b1_m^T\right)^T,\]
where $\b1_n$ and $\b1_m$ are $n \times 1$ and $m \times 1$ vectors of ones, respectively. Similar to the dependence-adjusted procedure in \citet{fan2012estimating}, the first two major principal factors from the estimated covariance matrix are removed to reduce the dependence among the test statistics.

With the dependence-adjusted $t$-test statistics, we estimate the asymptotic mean and standard deviation of FDP for three thresholds $t = 0.005, 0.02, 0.05$. For each threshold, the asymptotic mean is estimated based on Theorem \ref{Thm-4} and the asymptotic standard deviation is calculated based on the procedure in Section \ref{t-Sec-5.2}. In this analysis, we only use the ridge estimator of the covariance matrix, given that this estimator outperforms the sample covariance matrix in the simulation study.

The results are summarized in Table \ref{Table-9}. From this table, the ratio between the asymptotic standard deviation and the asymptotic mean of FDP is large, ranging from $0.6$ to $1.5$. Thus, the uncertainty of FDP can not be neglected even when the dependence among the test statistics has been reduced by removing the first two principal factors of the covariance matrix. In addition, SLIM yields much larger estimates of the asymptotic mean and variance of FDP than Langgas. In our analysis, SLIM estimates the proportion of true null hypotheses $\pi_0$ as $79.2\%$ and Langgas estimates $\pi_0$ as $52.3\%$, the former of which is a much more reasonable estimate. One explanation of such a difference is that Langaas' method assumes independence among the $p$-values while SLIM is more robust to dependence. The dependence among the test statistics deteriotes the performance of Langgas and affects the subsequent estimates for the asymptotic mean and variance of FDP.

\begin{table}
\centering
 \caption{Estimated asymptotic limit and standard deviation of FDP for real data}
\begin{tabular}{cccc}\\
\hline
t & \text{ Method } & $\widehat{\text{FDP}}(t)$ & $S_{\widehat{\text{FDP}}(t)}$\\ \hline
0.005 & SLIM & 4.89 & 7.28 \\
& Langaas & 2.21 & 3.32 \\ \hline
0.02 & SLIM & 12.53 & 11.63 \\
& Langaas & 5.23 & 5.12 \\ \hline
0.05 & SLIM & 22.91 & 13.84 \\
 & Langaas & 9.33 & 6.51 \\ \hline

\end{tabular}
\label{Table-9}
\end{table}

\section{Discussion}
 
One major contribution of this paper is to generalize the theoretical results on the asymptotic behavior of FDP from weakly dependent $z$-tests to weakly dependent $t$-tests. Similar to the findings in \citet{meng2021asymptotic}, FDP for weakly dependent $t$-tests converges almost surely to a deterministic limit, a quantity invariant to the test-statistic dependence; by contrast, the asymptotic uncertainty of FDP relies on the test-statistic dependence structure as implied in the asymptotic variance formulae in (\ref{V.1.unknown}) and (\ref{V.2.unknown}). Another major contribution of this paper is that we propose an approximation method to evaluate the asymptotic variance of FDP as the covariance of a pair of dependent $t$-tests has no closed form. Compared to the naive Monte Carlo method, the approximation method is computationally more efficient while producing a similarly accurate result.

As illustrated by the simulations and real data analysis, the dependence among the test statistics plays a critical role in determining the asymptotic uncertainty of FDP, which is not negligible compared to FDR for dependent $t$-tests. Therefore, it is important to account for the dependence when evaluating the asymptotic uncertainty of FDP. As a result, we recommend reporting the variance of FDP together with FDR when conducting a multiple hypothesis testing procedure. We also recommend the dependence-adjusted approaches to reduce dependence for multiple hypothesis testing procedures whenever possible.
 
In this paper, we focus on the inference of FDP for dependent $t$-tests when the population variances are unknown and need to be estimated. It will be equally interesting to see if these results can be generalized to other test frameworks such as large-sample tests in generalized linear models. In addition, we find that it is a challenging task to estimate the dependence between pairs of the test statistics in a high-dimensional setting, and the estimation errors could be accumulated in FDP or FDR estimation. Therefore, it will be an interesting problem to explore alternative methods to estimate the dependence between the test statistics and see how these methods affect the inference of FDP.

\bibliographystyle{asa}
\bibliography{FDP_t}

\end{document}

\maketitle

\section{Notation and Review of Technical Results in the Main Article}
Suppose $(X_{1},\ldots,X_{p})^{T} \sim N((\mu_{1},\ldots,\mu_{p})^{T},\bSigma)$ with unit variances, i.e., $\sigma_{jj}=1\text{ for } j=1,\ldots,p$. Define test statistics $\bT = \{T_1, \dots, T_p\} = \hat{\bD}^{-1}\sqrt{n}\Bar{\bX}$, which $\Bar{\bX}$ is the sample means of $\{\bX_i\}_{i = 1}^n$, and $\hat{\bD} = \text{diag}(\hat{\sigma}_1, \dots, \hat{\sigma}_p)$ are the square root diagonal elements of $\hat{\bSigma}$, the sample variance covariance matrix of  $\{\bX_i\}_{i = 1}^n$, and testing indicator result to be $t_j = \mathrm{1}(|T_j| > |q_{t/2}|)$, where $q_{t/2}$ is the $t/2$-quantile for $t_{n-1}$ distribution. Thus, the probability of simultaneously detecting signals will be
\begin{equation*}
\begin{split}
    & P(|T_{i_1}| > |q_{t/2}|, \dots, |T_{i_k}| > |q_{t/2}|)\\
    = & E(t_{i_1}t_{i_2}\dots t_{i_k}) \\
    = & E(E(t_{i_1}t_{i_2}\dots t_{i_k})|\{\hat{\sigma}_{i_1},\hat{\sigma}_{i_2},\dots,\hat{\sigma}_{i_k}\})\\
    = & E(P(|Z_{i_1}| > \hat{\sigma}_{i_1}|q_{t/2}|, \dots, |Z_{i_k}| > \hat{\sigma}_{i_k}|q_{t/2}|)|\{\hat{\sigma}_{i_1},\hat{\sigma}_{i_2},\dots,\hat{\sigma}_{i_k}\}) \\
    = & E(P^{\ast}(|Z_{i_1}| > \hat{\sigma}_{i_1}|q_{t/2}|, \dots, |Z_{i_k}| > \hat{\sigma}_{i_k}|q_{t/2}|)), \\
\end{split}
\end{equation*}
where  $(n-1)\hat{\sigma}^2_j \sim \chi_{n-1}^2$, and $P^{\ast}(\cdot) = P(\cdot|\{\text{All }\hat{\sigma}\})$.

We need the following notation in our development. Throughout, we use ``$\lesssim$" (``$\gtrsim$") to denote smaller (greater) than up to a universal constant. For any $t\in[0,1]$, we denote $\xi_j = P(|T_j|>|q_{t/2}|)$, and $\bar\xi = \frac{1}{p_1}\sum_{j\in \mathcal{H}_1} \xi_j$. Consider the function 
\begin{equation*}
H(\mu) =  E(\phi(q_{t/2}\hat{\sigma}_j)\hat{\sigma}_j \left\{ \phi(|q_{t/2}|\hat{\sigma} + \mu)(|q_{t/2}|\hat{\sigma} + \mu) + \phi(|q_{t/2}|\hat{\sigma} - \mu)(|q_{t/2}|\hat{\sigma} - \mu)\right\}); 
\end{equation*}
there exists a unique root, denoted by $\mu_t$, of this function for $\mu\in (0, \infty)$. Let $C_t^{\max} = \sup_{\mu \in (-\mu_t, \mu_t)} H(\mu)$. 

Let 
\begin{eqnarray*}f_{\{t_1,  \dots, t_k\}}^{(i_1,  \dots, i_k)}(a, t, \hat{\sigma}_i) = \dfrac{\partial^{\sum_{j = 1}^{k}i_j} f_{\{t_1, \ldots, t_k\}}( \rho_1, \ldots, \rho_k; a, t, \hat{\sigma}_i)}{\partial (\rho_1)^{i_1}   \ldots \partial(\rho_k)^{i_k}}\Big|_{\rho_j = 0, j=1,\ldots, k},
\end{eqnarray*}
where 
$$ f_{\{t_1,  \ldots, t_k\}}( \rho_1, \ldots, \rho_k; a, t, \hat{\sigma}_i) = \Phi \left(\dfrac{\hat{\sigma}_i q_{t/2} - \mu_a - \rho_1x_{t_1} -  \ldots - \rho_kx_{t_k}}{\sqrt{1 - \rho_1^2  - \ldots -\rho_k^2}}\right).$$
Let $\xi_i = E(t_i)$ for $i\in \mathcal{H}_1$; clearly we have $\xi_i = E(\Phi(q_{t/2}\hat{\sigma}_i + \mu_i) + \Phi(q_{t/2}\hat{\sigma}_i -\mu_i))$.

The following theorems are Theorem 1 and Theorem 2 in the main article; it establishes the asymptotic convergence and expansion of the $\FDP(t)$ in the framework of the weak dependence. 

\begin{theorem} \label{Thm-4}
Suppose $(X_{1},\ldots,X_{p})^{T} \sim N((\mu_{1},\ldots,\mu_{p})^{T},\bSigma)$ with unit variances, i.e., $\sigma_{jj}=1\text{ for } j=1,\ldots,p$. Define test statistics $\bT = \{T_1, \dots, T_p\} = \hat{\bD}^{-1}\sqrt{n}\Bar{\bX}$, which $\Bar{\bX}$ is the sample means of $\{\bX_i\}_{i = 1}^n$, and $\hat{\bD} = \text{diag}(\hat{\sigma}_1, \dots, \hat{\sigma}_p)$ are the square root diagonal elements of $\hat{\bSigma}$, the sample variance covariance matrix of  $\{\bX_i\}_{i = 1}^n$, and testing indicator result to be $t_j = \mathrm{1}(|T_j| > |q_{t/2}|)$. Assume that $(X_{1},\dots,X_{p})^T$ are weakly dependent as defined in (1) in the main article, that $\limsup_{p\to \infty} p_0t/(p_1\bar \xi) < 1$, and that as $p$ is sufficiently large. We have the following asymptotic convergence of $\FDP(t)$:
\begin{equation} \label{asymptotic.converge}
\lim_{p\to\infty} \left[\mathrm{FDP}(t) -\frac{\sum_{i \in \mathcal{H}_0}\xi_i}{\sum_{i=1}^{p}\xi_i}\right] = 0.
\end{equation}
\end{theorem}

\begin{theorem} \label{Thm-5}
Suppose $(X_{1},\ldots,X_{p})^{T} \sim N((\mu_{1},\ldots,\mu_{p})^{T},\bSigma)$ with unit variances, i.e., $\sigma_{jj}=1\text{ for } j=1,\ldots,p$. Define test statistics $\bT = \{T_1, \dots, T_p\} = \hat{\bD}^{-1}\sqrt{n}\Bar{\bX}$, which $\Bar{\bX}$ is the sample means of $\{\bX_i\}_{i = 1}^n$, and $\hat{\bD} = \text{diag}(\hat{\sigma}_1, \dots, \hat{\sigma}_p)$ are the square root diagonal elements of $\hat{\bSigma}$, the sample variance covariance matrix of  $\{\bX_i\}_{i = 1}^n$, and testing indicator result to be $t_j = \mathrm{1}(|T_j| > |q_{t/2}|)$. Assume that $(X_{1},\dots,X_{p})^T$ are weakly dependent as defined in (1) in the main article, that $\limsup_{p\to \infty} p_0t/(p_1\bar \xi) < 1$, and that as $p$ is sufficiently large, for a universal constant $C>0$,
\begin{align}
&\sum_{i\neq j; i,j\in \mathcal{H}_0} \sigma_{ij}^2 \geq \frac{C_t^{\max}}{E^2(\phi(\hat{\sigma}_i q_{t/2})|\hat{\sigma}_iq_{t/2}|)} \sum_{i\in \mathcal{H}_1, j\in \mathcal{H}_0, \mu_i\in[-\mu_t, \mu_t]} \sigma_{ij}^2, \\
&\sum_{i\neq j; i,j\in \mathcal{H}_0} \sigma_{ij}^2 + p \geq C \sum_{i\neq j; i, j\in \mathcal{H}_1} \sigma_{ij}^2, \\
&\sum_{i\neq j} \sigma_{ij}^4 = o\left(\sum_{i\neq j; i,j\in \mathcal{H}_0} \sigma_{ij}^2  + p_0\right).
\end{align}
We have the following asymptotic expansion of $\FDP(t)$:
\begin{equation} 
\FDP(t) = \frac{E(\bar V)}{E(\bar R)} + m(\bar{V},\bar{R}) +r(\bar V, \bar R), 
\end{equation}
where $\bar V = V(t)/p$, $\bar R = R(t)/p$,
\begin{equation*}
m(\bar{V},\bar{R}) =  \frac{\bar{V}}{E(\bar{R})} - \frac{E(\bar{V})}{\{E(\bar{R})\}^2} \bar{R},
\end{equation*}
and the remainder term $r(\bar V, \bar R)$ satisfies that $E\{r^2(\bar V, \bar R)\} = o[\Var\{m(\bar{V},\bar{R})\}]$.
\end{theorem}

A direct outcome of this theorem is the following Corollary, which establishes the explicit formula for the asymptotic variance of $\FDP(t)$. It is presented as Corollary 1 in the main article. 

\begin{corollary} \label{Corrollary-asym-var-unknown}
With all the conditions in Theorem \ref{Thm-5} effective, we have
\begin{equation} 
\lim_{p\to\infty}\dfrac{\Var\left\{\FDP(t)\right\}}{V_1(t) + V_2(t)} = 1, 
\end{equation}
where 
\begin{align}
V_1(t) ={} & \dfrac{p_1^2 \bar{\xi}^2}{(p_0 t + p_1\bar{\xi})^4} \times p_0t(1-t) + \dfrac{p_0^2 t^2}{(p_0 t + p_1\bar{\xi})^4} \times p_1\bar{\xi}(1 - \bar{\xi}), \\
V_2(t) ={} & \dfrac{2p_1^2 \bar{\xi}^2}{(p_0 t + p_1\bar{\xi})^4} \sum_{\substack{i<j \\ i,j\in\calH_0}} \Cov(t_i,t_j)  - \dfrac{2p_0p_1t\bar{\xi}}{(p_0 t + p_1\bar{\xi})^4} \sum_{\substack{i \in \calH_0 \\ j \in \calH_1}} \Cov(t_i, t_j) \notag \\
&+ \dfrac{2p_0^2 t^2}{(p_0 t + p_1\bar{\xi})^4} \sum_{\substack{i<j \\ i,j\in\calH_1}} \Cov(t_i,t_j). 
\end{align}
\end{corollary}

\section{Lemmas} \label{sec-lemmas-unknown}

We first establish the following lemmas; they play fundamental roles in our proofs of the theoretical results given above. In particular, Lemmas \ref{t-lem-3}--\ref{t-lem-7} are to support our proof for Theorem \ref{Thm-5}. 

% \begin{lemma}\label{t-lem-2}
% Denote $f_{\{t_1,  \dots, t_k\}}^{(i_1,  \dots, i_k)}(a, t) = \dfrac{\partial^{\sum_{j = 1}^{k}i_j} f_{\{t_1, \ldots, t_k\}}( \rho_1, \ldots, \rho_k; a, t)}{\partial (\rho_1)^{i_1}   \ldots \partial(\rho_k)^{i_k}}\Big|_{\rho_j = 0, j=1,\ldots, k}$, where 
% $$ f_{\{t_1,  \ldots, t_k\}}( \rho_1, \ldots, \rho_k; a, t) = \Phi \left(\dfrac{q_{t/2} - \mu_a - \rho_1x_{t_1} -  \ldots - \rho_kx_{t_k}}{\sqrt{1 - \rho_1^2  - \ldots -\rho_k^2}}\right),$$
% with $\rho_i \geq 0$ for all $i = 1, 2, \dots, k$. Then
% \begin{eqnarray} f_{\{t_1,  \dots, t_k\}}^{(i_1,  \dots, i_{k-1}, 0)}(a, t) = f_{\{t_1,  \dots, t_{k-1}\}}^{(i_1,  \dots, i_{k-1})}(a, t) \label{eq-f-partial}
% \end{eqnarray}
% \end{lemma}
% \begin{proof}
% Note that 
% $$f_{\{t_1, t_2, \ldots, t_k\}}( \rho_1, \ldots, \rho_k; a, t) = \Phi \left(\dfrac{(q_{t/2} - \mu_a -\rho_k x_{t_k})- \rho_1x_{t_1} -  \ldots - \rho_{k-1}x_{t_{k-1}}}{\sqrt{(1 - \rho_k^2)- \rho_1^2  - \ldots -\rho_{k-1}^2}}\right);$$
% when it is taken derivative with respect to $\rho_1, \ldots, \rho_{k-1}$, $q_{t/2} - \mu_a -\rho_k x_{t_k}$ and $1-\rho_k^2$ are the only terms involving $\rho_k$, and can be treated as constants. Based on the definition of $f_{\{t_1,  \dots, t_k\}}^{(i_1,  \dots, i_{k-1}, 0)}(a, t)$ and $f_{\{t_1,  \dots, t_{k-1}\}}^{(i_1,  \dots, i_{k-1})}(a, t)$, it is straightforward to check \eqref{eq-f-partial}. 
% \end{proof}

\begin{lemma}\label{t-lem-3}
The partial derivatives $f_{t_1, t_2, t_3}^{i_1, i_2, i_3}(a,t, \hat{\sigma}_i)$ can be written to be
$$ f_{t_1, t_2, t_3}^{i_1, i_2, i_3}(a,t, \hat{\sigma}_i) = C^{a}_{t/2, \hat{\sigma}_i}(i_1 + i_2 + i_3)g_{i_1}(x_{t_1})g_{i_2}(x_{t_2})g_{i_3}(x_{t_3}),$$
where $(n-1)\hat{\sigma}^2_i \sim \chi_{n-1}^2$ for $\forall i \in \{1, \dots, p\}$, and
\begin{eqnarray*}
&g_0(x) = 1, \qquad g_1(x) = x, \qquad g_2(x) = x^2 - 1, \qquad g_3(x) = x^3-3x,&\\ & g_4(x) = x^4 - 6x^2 + 3, \qquad  g_5(x) =  x^5 - 10x^3 + 15x& \\
&g_6(x) =  x^6 - 15x^4 + 45x^2 - 15, \qquad  g_7(x) =  x^7 - 21x^5 + 105x^3 - 105x&\\
&g_8(x) =  x^8 - 28x^6 + 210x^4 -420x^2 + 105,&
\end{eqnarray*}
$C^{a}_{t/2, \hat{\sigma}_i}(0) = \Phi(\hat{\sigma}_i q_{t/2} - \sqrt{n}\mu_a)$, and $C^{a}_{t/2, \hat{\sigma}_i}(i) = -\phi(\hat{\sigma}_i q_{t/2}- \sqrt{n}\mu_a) g_{i-1}(\hat{\sigma}_iq_{t/2} - \sqrt{n}\mu_a)$ for $i\geq 1$. Notice that $C^{a}_{t/2, \hat{\sigma}_i}(i)$ are bounded functions for $\forall i\in\{0, \dots, 9\}$ and $\forall \hat{\sigma}_i \in (0, \infty)$. Furthermore, for $X\sim N(0,1)$,
\begin{eqnarray*}
E\{g_i(X)\} = 0, && \mbox{ for } i=1,\ldots, 8,\\
\mbox{and } \quad E\{ g_{i_1}(X) g_{i_2}(X)\} = 0, && \mbox{ for  any }i_1, i_2\in \{0, \ldots, 8 \} , i_1 \neq i_2;
\end{eqnarray*}
and for $\mu_a = 0$,
\begin{eqnarray}
C^{a}_{1-t/2, \hat{\sigma}_i}(i) = \begin{cases}
C^{a}_{t/2, \hat{\sigma}_i}(i), &\text{ if i is odd}\\
-C^{a}_{t/2, \hat{\sigma}_i}(i), &\text{ if i is even}.\\ 
\end{cases} \label{t-eq-lem3-0}
\end{eqnarray}
\end{lemma}

\begin{proof}
The proof is based on straightforward but tedious evaluations on the partial derivatives of $f_{\{t_1, \ldots, t_k\}}(\rho_1, \ldots, \rho_k; a,t)$ with respect to $\rho_1,\ldots, \rho_k$; the details are omitted. 
\end{proof}

\begin{lemma}\label{t-lem-4}
Recall the definition:  $t_i = \mathrm{1}(|T_i| > |q_{t/2}|) = \mathrm{1}(P_i < t)$; we have
\begin{itemize}
    \item[(P1). ] when $i\neq j$, and $i,j\in \mathcal{H}_1$,  \begin{eqnarray*}
    &&E(t_it_j) - E(t_i) E(t_j) \\&=& E(\left\{C^i_{t/2, \hat{\sigma}_i}(1)-C^i_{1-t/2, \hat{\sigma}_i}(1)\right\}\left\{C^j_{t/2, \hat{\sigma}_j}(1)-C^j_{1-t/2, \hat{\sigma}_j}(1)\right\})\sigma_{ij}\\
     &&+ \frac{1}{2}E(\left\{C^i_{t/2, \hat{\sigma}_i}(2)-C^i_{1-t/2, \hat{\sigma}_i}(2)\right\}\left\{C^j_{t/2, \hat{\sigma}_j}(2)-C^j_{1-t/2, \hat{\sigma}_j}(2)\right\})\sigma_{ij}^2\\
     &&+\frac{E\{g_3^2(X_1)\}}{(3!)^2} E(\left\{C_{t/2, \hat{\sigma}_i}^i(3) -  C_{1-t/2, \hat{\sigma}_i}^i(3)\right\}\left\{C_{t/2, \hat{\sigma}_j}^j(3) -  C_{1-t/2, \hat{\sigma}_j}^j(3)\right\})\sigma_{ij}^3+O(\sigma_{ij}^4);
    \end{eqnarray*}
    
    \item[(P2).]  when $i\in \mathcal{H}_1, j\in \mathcal{H}_0$, for any $t$, there exists a unique root $\mu_t \in (0, \infty)$ of $H(\mu_j)$, such that when $\mu_j \leq \mu_t$,
    \begin{eqnarray*}
    && O(\sigma_{ij}^4)< E(t_it_j) - E(t_i) E(t_j) = -\sigma_{ij}^2 E(\phi(q_{t/2}\hat{\sigma}_j)\hat{\sigma}_jq_{t/2} \left\{ C_{t/2, \hat{\sigma}_j}^j(2)) - E(C_{1-t/2, \hat{\sigma}_i}^j(2) \right\}) + O(\sigma_{ij}^4) \\
    &\leq &  C_t^{\max} \sigma_{ij}^2 + O(\sigma_{ij}^4),
    \end{eqnarray*}
   where 
   \begin{eqnarray*}
    C_t^{\max} = \sup_{\mu}  E(\phi(q_{t/2}\hat{\sigma}_j)\hat{\sigma}_jq_{t/2} \left\{ C_{t/2, \hat{\sigma}_j}^j(2)) - E(C_{1-t/2, \hat{\sigma}_i}^j(2) \right\});
    \end{eqnarray*} 
    and when  $|\mu_j|\geq \mu_t$,
    \begin{eqnarray*}
    E(t_it_j) - E(t_i) E(t_j) < O(\sigma_{ij}^4).
    \end{eqnarray*}
    
    \item[(P3).] when $i\neq j$, and $i,j\in \mathcal{H}_0$, 
    \begin{eqnarray*}
    E(t_it_j) - E(t_i) E(t_j) = 2 E(\phi^2(q_{t/2}\hat{\sigma}_i)(q_{t/2}\hat{\sigma}_i)^2) \sigma_{ij}^2 + O(\sigma_{ij}^4).
    \end{eqnarray*}
\end{itemize}
\end{lemma}

\begin{proof}
Note that we can write 
\begin{eqnarray*}
E(t_it_j) - E(t_i) E(t_j) = \sum_{a_1, a_2 \in \{1, -1\}}a_1a_2 E(h_{i,j}(a_1, a_2)), \label{t-eq-lem4-1}
\end{eqnarray*}
where
\begin{eqnarray*}
    h_{i,j}(a_1, a_2) = P^{\ast}(Z_{i} < \hat{\sigma}_{i} a_1 q_{t/2}, Z_j < \hat{\sigma}_{j} a_2 q_{t/2})\label{t-eq-lem4-2}
\end{eqnarray*}

We derive $h_{i,j}(a_1, a_2)$ for $a_1 = a_2 = 1$; the derivations for other values of $a_1$ and $a_2$ are similar. Let $X_1, X_2, X_3$ be i.i.d. standard normal random variables, we can write 
\begin{eqnarray*}Z_i = \mu_i + \sqrt{|\sigma_{ij}|} X_1 + \sqrt{1-|\sigma_{ij}|}X_2 &&\\
Z_j = \mu_j + \sqrt{|\sigma_{ij}|} X_1 + \sqrt{1-|\sigma_{ij}|}X_3 &&\mbox{if } \sigma_{ij}\geq 0,
\end{eqnarray*}
and 
\begin{eqnarray*}
Z_j = \mu_j - \sqrt{|\sigma_{ij}|} X_1 + \sqrt{1-|\sigma_{ij}|}X_3 \quad \mbox{if } \sigma_{ij}<0.
\end{eqnarray*} 
Hereafter, we assume that $\sigma_{ij}\geq0$; the case for $\sigma_{ij}<0$ can be derived the same. 
The first term of $h_{i,j}(a_1, a_2)$ is given by 
\begin{equation}
    \begin{split}
        &  P^{\ast}(Z_{i} < \hat{\sigma}_{i} a_1 q_{t/2}, Z_j < \hat{\sigma}_{j} a_2 q_{t/2}) \\
        = & \int \Phi \left(\dfrac{\hat{\sigma}_{i} q_{t/2} - \mu_i - \sqrt{\sigma_{ij}}x_{1}}{\sqrt{1 - \sigma_{ij}}}\right)\Phi \left(\dfrac{\hat{\sigma}_{j} q_{t/2} - \mu_j - \sqrt{\sigma_{ij}}x_{1}}{\sqrt{1 - \sigma_{ij}}}\right)\phi(x_1)dx_1.\\
    \end{split} \label{t-eq-lem4-3}
\end{equation}
Note that by definition $\Phi \left(\dfrac{\hat{\sigma}_{i} q_{t/2} - \mu_a - \sqrt{\sigma_{ij}}x_{1}}{\sqrt{1 - \sigma_{ij}}}\right) = f_{\{1\}}(\sqrt{\sigma_{ij}}; a, t, \hat{\sigma}_{i})$; applying Taylor expansion and Lemma \ref{t-lem-3}, we immediately have
\begin{eqnarray}
\Phi \left(\dfrac{\hat{\sigma}_{i} q_{t/2} - \mu_a - \sqrt{\sigma_{ij}}x_{1}}{\sqrt{1 - \sigma_{ij}}}\right) = \sum_{i_1 \leq 7} C^a_{t/2, \hat{\sigma}_{i}}(i_1)g_{i_1}(x_1)(\sqrt{\sigma_{ij}})^{i_1}/(i_1!) + R(\sigma_{ij}), \label{t-eq-lem4-4}
\end{eqnarray}
where $R(\sigma_{ij})$ is the Lagrange residual term in the Taylor's expansion, satisfying $|R(\sigma_{ij})| \lesssim |f(x_1;\hat{\sigma}_{i})| \sigma_{ij}^4$ up to a univeral constant not depending on $x_1$, with $f(x_1;\hat{\sigma}_{i})$ being a finite order polynomial function of $x_1$, and the coefficients are expected value of bounded functions in term of $\hat{\sigma}_{i}$. Combining \eqref{t-eq-lem4-3} and \eqref{t-eq-lem4-4} leads to 
\begin{equation}
    \begin{split}
        &  P^{\ast}(Z_{i} < \hat{\sigma}_{i} a_1 q_{t/2}, Z_j < \hat{\sigma}_{j} a_2 q_{t/2}) \\
        = &\sum_{i_1 + i_2 \leq 7} C^i_{t/2,\hat{\sigma}_{i}}(i_1)C^j_{t/2, \hat{\sigma}_{j}}(i_2)(\sqrt{\sigma_{ij}})^{i_1 + i_2}E(g_{i_1}(X_1)g_{i_2}(X_1))/(i_1!i_2!) + O(\sigma_{ij}^4) \\
        = &\sum_{i_1 = 0}^3 C^i_{t/2, \hat{\sigma}_{i}}(i_1)C^j_{t/2, \hat{\sigma}_{j}}(i_1)\sigma_{ij}^{i_1}E(g_{i_1}^2(X_1))/(i_1!)^2 + O(\sigma_{ij}^4), \label{t-eq-lem4-5}\\
    \end{split}
\end{equation}
where the second ``=" is because $E(g_{i_1}(X_1)g_{i_2}(X_1)) = 0$ for $i_1\neq i_2$ based on Lemma \ref{t-lem-3}; $X_1$ is a standard normal random variable. Furthermore, it is straightforward to check that for $i_1 = 0$,
\begin{eqnarray*}
C^i_{t/2, \hat{\sigma}_{i}}(i_1)C^j_{t/2, \hat{\sigma}_{j}}(i_1)\sigma_{ij}^{i_1}E(g_{i_1}^2(X_1))/(i_1!)^2 = P^{\ast}(Z_i < \hat{\sigma}_{i}q_{t/2})P^{\ast}(Z_{j} < \hat{\sigma}_{j}q_{t/2}),
\end{eqnarray*}
which together with \eqref{t-eq-lem4-5} leads to
\begin{eqnarray*}
h_{i,j}(1,1) &=&  P^{\ast}(Z_{i} < \hat{\sigma}_{i} a_1 q_{t/2}, Z_j < \hat{\sigma}_{j} a_2 q_{t/2}) -  P^{\ast}(Z_{i} < \hat{\sigma}_{i} a_1 q_{t/2}) P^{\ast}(Z_j < \hat{\sigma}_{j} a_2 q_{t/2})\\
&=& \sum_{i_1 = 1}^3 C^i_{t/2, \hat{\sigma}_{i}}(i_1)C^j_{t/2, \hat{\sigma}_{j}}(i_1)\sigma_{ij}^{i_1}E(g_{i_1}^2(X_1))/(i_1!)^2 + O(\sigma_{ij}^4). 
\end{eqnarray*}
Similarly, we can derive $h_{i,j}(1,-1)$, $h_{i,j}(-1,1)$, and $h_{i,j}(-1,-1)$. As consequence, we have
\begin{equation*}
    \begin{split}
        &E(t_it_j) - E(t_i) E(t_j)= \sum_{a_1, a_2 \in \{1, -1\}}a_1a_2 E(h_{i,j}(a_1, a_2))\\
        = & E(\sum_{i_1 = 1}^3 C^i_{t/2, \hat{\sigma}_{i}}(i_1)C^j_{t/2, \hat{\sigma}_{j}}(i_1)\sigma_{ij}^{i_1}E\{g_{i_1}^2(X_1)\}/(i_1!)^2\\
        - & \sum_{i_1 = 1}^3 C^i_{1-t/2, \hat{\sigma}_{i}}(i_1)C^j_{t/2,\hat{\sigma}_{j}}(i_1)\sigma_{ij}^{i_1}E\{g_{i_1}^2(X_1)\}/(i_1!)^2\\
        - &\sum_{i_1 = 1}^3 C^i_{t/2, \hat{\sigma}_{i}}(i_1)C^j_{1-t/2, \hat{\sigma}_{j}}(i_1)\sigma_{ij}^{i_1}E\{g_{i_1}^2(X_1)\}/(i_1!)^2\\
        + &\sum_{i_1 = 1}^3 C^i_{1-t/2, \hat{\sigma}_{i}}(i_1)C^j_{1-t/2, \hat{\sigma}_{j}}(i_1)\sigma_{ij}^{i_1}E\{g_{i_1}^2(X_1)\}/(i_1!)^2) + O(\sigma_{ij}^4)\\
        = & \sum_{i_1 = 1}^3 \frac{\sigma_{ij}^{i_1} E\{g_{i_1}^2(X_1)\}}{(i_1!)^2} E(\left\{C^i_{t/2, \hat{\sigma}_{i}}(i_1)-C^i_{1-t/2, \hat{\sigma}_{i}}(i_1)\right\}\left\{C^j_{t/2, \hat{\sigma}_{j}}(i_1)-C^j_{1-t/2, \hat{\sigma}_{j}}(i_1)\right\})+O(\sigma_{ij}^4).
\end{split} 
\end{equation*}
We are able to check:
\begin{itemize}

   \item when $\mu_i = 0$, for $i_1 = 1$ or $3$, 
   \begin{eqnarray*}
    C^i_{t/2, \hat{\sigma}_{i}}(i_1)-C^i_{1-t/2, \hat{\sigma}_{i}}(i_1) = 0,
    \end{eqnarray*}
    and for $i_1 = 2$ 
     \begin{eqnarray*}
    C^i_{t/2, \hat{\sigma}_{i}}(i_1)-C^i_{1-t/2, \hat{\sigma}_{i}}(i_1) = -2 \phi(\hat{\sigma}_{i}q_{t/2}) \hat{\sigma}_{i}q_{t/2};
    \end{eqnarray*}
    
    \item when $\mu_i \neq 0$, for $i_1 = 1$,
    \begin{eqnarray*}
    C^i_{t/2, \hat{\sigma}_{i}}(i_1)-C^i_{1-t/2, \hat{\sigma}_{i}}(i_1) = -\phi(\hat{\sigma}_{i}|q_{t/2}|+\mu_i) + \phi(\hat{\sigma}_{i}|q_{t/2}|-\mu_i);
    \end{eqnarray*}
    for $i_1 = 2$,
    \begin{eqnarray*}
    C^i_{t/2, \hat{\sigma}_{i}}(2)-C^i_{1-t/2, \hat{\sigma}_{i}}(2) &=& -\phi(\hat{\sigma}_{i}q_{t/2}- \mu_i)(\hat{\sigma}_{i}q_{t/2} - \mu_i) - \phi(\hat{\sigma}_{i}q_{t/2} + \mu_i)(\hat{\sigma}_{i}q_{t/2} + \mu_i) \\
    &=& \phi(\hat{\sigma}_{i}|q_{t/2}|+ \mu_i)(\hat{\sigma}_{i}|q_{t/2}| + \mu_i) + \phi(\hat{\sigma}_{i}|q_{t/2}| - \mu_i)(\hat{\sigma}_{i}|q_{t/2}| - \mu_i).
    \end{eqnarray*}

    and for $i_1 = 3$,
    \begin{eqnarray*}
    &&C^i_{t/2, \hat{\sigma}_{i}}(3)-C^i_{1-t/2, \hat{\sigma}_{i}}(3) \\ &=& - \phi(\hat{\sigma}_{i} q_{t/2}-\mu_i)\left\{ (\hat{\sigma}_{i}q_{t/2}-\mu_i)^2 - 1\right\} + \phi(\hat{\sigma}_{i}q_{t/2}+\mu_i)\left\{ (\hat{\sigma}_{i}q_{t/2}+\mu_i)^2 -1 \right\}\\
    &=&  - \phi(\hat{\sigma}_{i}|q_{t/2}|+\mu_i)\left\{ (\hat{\sigma}_{i}|q_{t/2}|+\mu_i)^2 - 1\right\} + \phi(\hat{\sigma}_{i}|q_{t/2}|-\mu_i)\left\{ (\hat{\sigma}_{i}|q_{t/2}|-\mu_i)^2 -1 \right\}
    \end{eqnarray*}
    
    \end{itemize}
    
    We can conclude:
    
    \begin{itemize}

    \item when $\mu_i \neq 0$ and $\mu_j \neq 0$, 
    \begin{eqnarray*}
    &&E(t_it_j) - E(t_i) E(t_j) \\&=& E(\left\{C^i_{t/2, \hat{\sigma}_i}(1)-C^i_{1-t/2, \hat{\sigma}_i}(1)\right\}\left\{C^j_{t/2, \hat{\sigma}_j}(1)-C^j_{1-t/2, \hat{\sigma}_j}(1)\right\})\sigma_{ij}\\
     &&+ \frac{1}{2}E(\left\{C^i_{t/2, \hat{\sigma}_i}(2)-C^i_{1-t/2, \hat{\sigma}_i}(2)\right\}\left\{C^j_{t/2, \hat{\sigma}_j}(2)-C^j_{1-t/2, \hat{\sigma}_j}(2)\right\})\sigma_{ij}^2\\
     &&+\frac{E\{g_3^2(X_1)\}}{(3!)^2} E(\left\{C_{t/2, \hat{\sigma}_i}^i(3) -  C_{1-t/2, \hat{\sigma}_i}^i(3)\right\}\left\{C_{t/2, \hat{\sigma}_j}^j(3) -  C_{1-t/2, \hat{\sigma}_j}^j(3)\right\})\sigma_{ij}^3+O(\sigma_{ij}^4);
    \end{eqnarray*}
    which verifies (P1). 
    
    \item when $\mu_i \neq 0$ but $\mu_j = 0$, 
    \begin{eqnarray*}
    && E(t_it_j) - E(t_i) E(t_j) = -\sigma_{ij}^2 q_{t/2}E(\phi(q_{t/2}\hat{\sigma}_j)\hat{\sigma}_j \left\{ C_{t/2, \hat{\sigma}_j}^j(2)) - E(C_{1-t/2, \hat{\sigma}_i}^j(2) \right\}) + O(\sigma_{ij}^4). \\
    \end{eqnarray*}
    Since $C_{t/2, \hat{\sigma}_j}^j(i))$ is a bounded function for $\forall t \in (0, 1)$, $\hat{\sigma}_j \in (0, \infty)$ and $i \in \{1, \dots, 9\}$. Thus, there exist a 
    $$C_t^{\max} = \sup_{\mu} \left\{q_{t/2}E(\phi(q_{t/2}\hat{\sigma}_j)\hat{\sigma}_j \left\{ C_{t/2, \hat{\sigma}_j}^j(2)) - E(C_{1-t/2, \hat{\sigma}_i}^j(2) \right\})\right\}.$$
    Moreover, there exist a unique root $\mu_t \in (0, \infty)$ of 
    $$H(\mu_j) = E(\phi(q_{t/2}\hat{\sigma}_j)\hat{\sigma}_j \left\{ C_{t/2, \hat{\sigma}_j}^j(2)) - E(C_{1-t/2, \hat{\sigma}_i}^j(2) \right\}),$$ 
    such that when $\mu_j \in (-\mu_t, \mu_t)$, $C_t^{\max} \geq H(\mu_j) > 0$, where 
    \begin{eqnarray*}
    O(\sigma_{ij}^4)< E(t_it_j) - E(t_i) E(t_j) \leq   C_t^{\max} \sigma_{ij}^2 + O(\sigma_{ij}^4), 
    \end{eqnarray*}
    and 
    \begin{eqnarray*}
    E(t_it_j) - E(t_i) E(t_j)< O(\sigma_{ij}^4),
    \end{eqnarray*}
    otherwise. Therefore, (P2) is proved.

    \item when $\mu_i = 0$ and $\mu_j = 0$,
    \begin{eqnarray*}
    E(t_it_j) - E(t_i) E(t_j) = 2 E(\phi^2(q_{t/2}\hat{\sigma}_i)(q_{t/2}\hat{\sigma}_i)^2) \sigma_{ij}^2 + O(\sigma_{ij}^4),
    \end{eqnarray*}
    which is our (P3) and we complete the proof of this lemma. 
\end{itemize}
\end{proof}

\begin{lemma}\label{t-lem-6}
Recall the definition:  $t_i = \mathrm{1}(|T_i| > |q_{t/2}|) = \mathrm{1}(P_i < t)$; we have  
\begin{equation*}
    \begin{split}
        &\sum_{i\neq j\neq k \neq l} \Big\{E(t_i t_j t_k t_l) - E(t_i t_j t_k)E(t_l) - E(t_i t_j t_l)E(t_k) - E(t_i t_k t_l)E(t_j) - E(t_j t_k t_l)E(t_i)\\
        + &E(t_i t_j)E(t_k)E(t_l) + E(t_i t_k)E(t_j)E(t_l) + E(t_i t_l)E(t_j)E(t_k) + E(t_j t_k)E(t_i)E(t_l)\\
        + &E(t_j t_l)E(t_i)E(t_k) + E(t_k t_l)E(t_i)E(t_j) -3E(t_i)E(t_j)E(t_k)E(t_l)\Big\}\\
        = & 3 E\Big\{\left[\sum_{i\neq j; i,j\in \mathcal{H}_1}\left\{C_{t/2, \hat{\sigma}_i}^i(1) -  C_{1-t/2, \hat{\sigma}_i}^i(1)\right\}\left\{C_{t/2, \hat{\sigma}_j}^j(1) -  C_{1-t/2, \hat{\sigma}_j}^j(1)\right\}\sigma_{ij}\right]^2\Big\} \\  &+ 3E\Big\{\left[\sum_{i\neq j; i,j\in \mathcal{H}_1}\left\{C_{t/2, \hat{\sigma}_i}^i(1) -  C_{1-t/2, \hat{\sigma}_i}^i(1)\right\}\left\{C_{t/2, \hat{\sigma}_j}^j(1) -  C_{1-t/2, \hat{\sigma}_j}^j(1)\right\}\sigma_{ij}\right]\\
    &\times \left[\sum_{i\neq j}\left\{C_{t/2, \hat{\sigma}_i}^i(2) -  C_{1-t/2, \hat{\sigma}_i}^i(2)\right\}\left\{C_{t/2, \hat{\sigma}_j}^j(2) -  C_{1-t/2, \hat{\sigma}_j}^j(2)\right\}\sigma_{ij}^2\right]\Big\} \\
        & + O\left( \sum_{i\in \mathcal{H}_1}\left( \sum_{j\neq i, j\in\mathcal{H}_1} |\sigma_{ij}|\right)^3 \right) + O\left( \sum_{i\neq j} |\sigma_{ij}| \sum_{k\neq i, k\in \mathcal{H}_1} |\sigma_{ik}|\sum_{l\neq j, l\in \mathcal{H}_1}|\sigma_{jl}| \right)+ O\left( p^2 \sum_{i\neq j}\sigma_{ij}^4\right). \\
    \end{split}
\end{equation*}
\end{lemma}

\begin{proof}
We can write
\begin{equation}
    \begin{split}
        &E(t_i t_j t_k t_l) - E(t_i t_j t_k)E(t_l) - E(t_i t_j t_l)E(t_k) - E(t_i t_k t_l)E(t_j) - E(t_j t_k t_l)E(t_i)\\
        + &E(t_i t_j)E(t_k)E(t_l) + E(t_i t_k)E(t_j)E(t_l) + E(t_i t_l)E(t_j)E(t_k) + E(t_j t_k)E(t_i)E(t_l)\\
        + &E(t_j t_l)E(t_i)E(t_k) + E(t_k t_l)E(t_i)E(t_j) -3E(t_i)E(t_j)E(t_k)E(t_l)\\
        = &  \sum_{a_1, a_2, a_3, a_4 = \{1, -1\}}a_1 a_2 a_3 a_4 E(h_{i, j, k, l; t}(a_1, a_2, a_3, a_4)),
    \end{split} \label{t-eq-lem6-1}
\end{equation}
where $h_{i, j, k, l}(a_1, a_2, a_3, a_4) = P^{\ast}(Z_i < a_1 \hat{\sigma}_i q_{t/2}, Z_j < a_2 \hat{\sigma}_j q_{t/2}, Z_k < a_3 \hat{\sigma}_k q_{t/2}, Z_l < a_4 \hat{\sigma}_l q_{t/2}) - P^{\ast}(Z_i < a_1 \hat{\sigma}_i q_{t/2}, Z_j < a_2\hat{\sigma}_j q_{t/2}, Z_k < a_3 \hat{\sigma}_k q_{t/2})P^{\ast}(Z_l < a_4 \hat{\sigma}_l q_{t/2}) - P^{\ast}(Z_i < a_1 \hat{\sigma}_i q_{t/2}, Z_j < a_2 \hat{\sigma}_j q_{t/2}, Z_l < a_4 \hat{\sigma}_l q_{t/2})P^{\ast}(Z_k < a_3 \hat{\sigma}_k q_{t/2}) - P^{\ast}(Z_i < \hat{\sigma}_i a_1 q_{t/2}, Z_k < a_3\hat{\sigma}_k q_{t/2}, Z_l < a_4 \hat{\sigma}_l q_{t/2})P^{\ast}(Z_j < a_2 \hat{\sigma}_j q_{t/2}) - P^{\ast}(Z_j < a_2 \hat{\sigma}_j q_{t/2}, Z_k < a_3 \hat{\sigma}_k q_{t/2}, Z_l < a_4 \hat{\sigma}_l q_{t/2})P^{\ast}(Z_i < a_1 \hat{\sigma}_i q_{t/2}) + P^{\ast}(Z_i < a_1 \hat{\sigma}_i q_{t/2}, Z_j < a_2 \hat{\sigma}_j q_{t/2})P^{\ast}(Z_k < a_3 \hat{\sigma}_k q_{t/2})P^{\ast}(Z_l < a_4 \hat{\sigma}_l q_{t/2}) + P^{\ast}(Z_i < a_1 \hat{\sigma}_i q_{t/2}, Z_k < a_3 \hat{\sigma}_k q_{t/2})P^{\ast}(Z_j < a_2 \hat{\sigma}_j q_{t/2})P^{\ast}(Z_l < a_4 \hat{\sigma}_l q_{t/2}) + P^{\ast}(Z_i < a_1 \hat{\sigma}_i q_{t/2}, Z_l < a_4 \hat{\sigma}_l q_{t/2})P^{\ast}(Z_j < a_2 \hat{\sigma}_j q_{t/2})P^{\ast}(Z_k < a_3 \hat{\sigma}_k q_{t/2}) + P^{\ast}(Z_j < a_2 \hat{\sigma}_j q_{t/2}, Z_k < a_3 \hat{\sigma}_k q_{t/2})P^{\ast}(Z_i < a_1 \hat{\sigma}_i q_{t/2})P^{\ast}(Z_l < a_4 \hat{\sigma}_l q_{t/2}) + P^{\ast}(Z_j < a_2 \hat{\sigma}_j q_{t/2}, Z_l < a_4 \hat{\sigma}_l q_{t/2})P^{\ast}(Z_i < a_1 \hat{\sigma}_i q_{t/2})P^{\ast}(Z_k < a_3 \hat{\sigma}_k q_{t/2}) + P^{\ast}(Z_k < a_3 \hat{\sigma}_k q_{t/2}, Z_l < a_4 \hat{\sigma}_l q_{t/2})P^{\ast}(Z_i < a_1 \hat{\sigma}_i q_{t/2})P^{\ast}(Z_j < a_2 \hat{\sigma}_j q_{t/2}) - 3P^{\ast}(Z_i < a_1\hat{\sigma}_i q_{t/2})P^{\ast}(Z_j < a_2 \hat{\sigma}_j q_{t/2})P^{\ast}(Z_k < a_3 \hat{\sigma}_k q_{t/2})P^{\ast}(Z_l < a_4 \hat{\sigma}_l q_{t/2})$

Next, we consider the terms in $h_{i,j,k,l}(1,1,1,1)$; those for other values of $a_1,\ldots, a_4$ can be similarly derived. Without loss of generality, we assume $\sigma_{ij}\geq0$ for $i=1,\ldots, p; j=1, \ldots, p$. Note that $h_{i,j,k,l}(1,1,1,1)$  contains 16 terms with the last three terms identical, and its first term is given by
\begin{equation}
    \begin{split}
         & P^{\ast}(Z_i < a_1 \hat{\sigma}_i q_{t/2}, Z_j < a_2 \hat{\sigma}_j q_{t/2}, Z_k < a_3 \hat{\sigma}_k q_{t/2}, Z_l < a_4 \hat{\sigma}_l q_{t/2})\\
         = & \int \Phi \left(\dfrac{q_{t/2}\hat{\sigma}_i - \mu_i - \sqrt{\sigma_{ij}}x_{1} - \sqrt{\sigma_{ik}}x_{2} - \sqrt{\sigma_{il}}x_{3}}{\sqrt{1 - \sigma_{ij} - \sigma_{ik} - \sigma_{il}}}\right)\Phi\left (\dfrac{q_{t/2}\hat{\sigma}_j - \mu_j - \sqrt{\sigma_{ij}}x_{1} - \sqrt{\sigma_{jk}}x_{4} - \sqrt{\sigma_{jl}}x_{5}}{\sqrt{1 - \sigma_{ij} - \sigma_{jk} - \sigma_{jl}}}\right)\\
         &\Phi \left(\dfrac{q_{t/2}\hat{\sigma}_k - \mu_k - \sqrt{\sigma_{ik}}x_{2} - \sqrt{\sigma_{jk}}x_{4} - \sqrt{\sigma_{kl}}x_{6}}{\sqrt{1 - \sigma_{ik} - \sigma_{jk} - \sigma_{kl}}}\right)\Phi\left (\dfrac{q_{t/2}\hat{\sigma}_l - \mu_l - \sqrt{\sigma_{il}}x_{3} - \sqrt{\sigma_{jl}}x_{5} - \sqrt{\sigma_{kl}}x_{6}}{\sqrt{1 - \sigma_{il} - \sigma_{jl} - \sigma_{kl}}}\right)\\
         &\phi(x_1)\phi(x_2)\phi(x_3)\phi(x_4)\phi(x_5) \phi(x_6)dx_1 dx_2 dx_3 dx_4 dx_5 dx_6.\\
    \end{split} \label{t-eq-lem6-2}
\end{equation}
Applying Taylor expansion and Lemma \ref{t-lem-3}, we have
\begin{equation*}
    \begin{split}
        &\Phi \left(\dfrac{q_{t/2}\hat{\sigma}_i - \mu_i - \sqrt{\sigma_{ij}}x_{1} - \sqrt{\sigma_{ik}}x_{2} - \sqrt{\sigma_{il}}x_{3}}{\sqrt{1 - \sigma_{ij} - \sigma_{ik} - \sigma_{il}}}\right)\\
        = & \sum_{i_1 + i_2 + i_3 \leq 7}C^{i}_{t/2, \hat{\sigma}_i}(i_1 + i_2 + i_3)g_{i_1}(x_1)g_{i_2}(x_2)g_{i_3}(x_3)(\sqrt{\sigma_{ij}})^{i_1}(\sqrt{\sigma_{ik}})^{i_2}(\sqrt{\sigma_{il}})^{i_3}/(i_1 + i_2 + i_3)! + R(\vec{\rho}),
    \end{split}
\end{equation*}
where $R(\vec{\rho})$ is the Lagrange residual term in the Taylor's expansion, and $|R(\vec{\rho})| \lesssim |f(x_1, x_2, x_3; \hat{\sigma}_i)|(|\sigma_{ij}|^4 + |\sigma_{ik}|^4 + |\sigma_{il}|^4)$ up to a universal constant not depending on $x_1, x_2, x_3$, where
$f(x_1, x_2, x_3; \hat{\sigma}_i)$ is a finite order polynomial function of $\{x_1, x_2, x_3\}$ with bounded coefficient for $\forall \hat{\sigma}_i \in (0, \infty)$. Similarly, we can apply Taylor expansion to other $\Phi(\cdot)$ terms in \eqref{t-eq-lem6-2}; we obtain
\begin{equation}\label{t-eq-lem6-3}
    \begin{split}
         & P^{\ast}(Z_i < a_1 \hat{\sigma}_i q_{t/2}, Z_j < a_2 \hat{\sigma}_j q_{t/2}, Z_k < a_3 \hat{\sigma}_k q_{t/2}, Z_l < a_4 \hat{\sigma}_l q_{t/2})\\
         = &  \int \sum_{\sum_{j = 1}^{12} i_j \leq 7}C^{i}_{t/2, \hat{\sigma}_i}(i_1 + i_2 + i_3)C^{j}_{t/2, \hat{\sigma}_j}(i_7 + i_4 + i_5)C^{k}_{t/2, \hat{\sigma}_k}(i_8 + i_{10} + i_6)C^{l}_{t/2, \hat{\sigma}_l}(i_{9} + i_{11} + i_{12})\\
         & g_{i_1}(x_1) g_{i_7}(x_1)g_{i_2}(x_2)g_{i_8}(x_2)g_{i_3}(x_3)g_{i_{9}}(x_3)g_{i_{4}}(x_4)g_{i_{10}}(x_4)g_{i_{5}}(x_5)g_{i_{11}}(x_5)g_{i_{6}}(x_6)g_{i_{12}}(x_6)\\
        &(\sqrt{\sigma_{ij}})^{i_1 + i_7}(\sqrt{\sigma_{ik}})^{i_2 + i_8}(\sqrt{\sigma_{il}})^{i_3 + i_9}(\sqrt{\sigma_{jk}})^{i_4 + i_{10}}(\sqrt{\sigma_{jl}})^{i_5 + i_{11}}(\sqrt{\sigma_{kl}})^{i_6 + i_{12}}\\
        &/((i_1 + i_2 + i_3)!(i_7 + i_4 + i_5)!(i_8 + i_{10} + i_6)!(i_{9} + i_{11} + i_{12})!)\\
        &\phi(x_1)\phi(x_2)\phi(x_3)\phi(x_4)\phi(x_5)\phi(x_6)dx_1 dx_2 dx_3 dx_4 dx_5 dx_6\\
        &+ O\left(|\sigma_{ij}|^4 + |\sigma_{ik}|^4 +|\sigma_{il}|^4 + |\sigma_{jk}|^4 + |\sigma_{jl}|^4 + |\sigma_{kl}|^4\right).
    \end{split}
\end{equation}
Applying Lemma \ref{t-lem-3}, \eqref{t-eq-lem6-3} can be further simplified to be
\begin{eqnarray}
&&P^{\ast}(Z_i < a_1 \hat{\sigma}_i q_{t/2}, Z_j < a_2 \hat{\sigma}_j q_{t/2}, Z_k < a_3 \hat{\sigma}_k q_{t/2}, Z_l < a_4 \hat{\sigma}_l q_{t/2}) \nonumber \\
&=& \sum_{\sum_{j=1}^6 i_j \in \{0, 1, 2, 3\} } C^{i}_{t/2, \hat{\sigma}_i}(i_1 + i_2 + i_3)C^{j}_{t/2, \hat{\sigma}_j}(i_1 + i_4 + i_5)C^{k}_{t/2, \hat{\sigma}_k}(i_2 + i_4 + i_6)C^{l}_{t/2, \hat{\sigma}_l}(i_3 + i_5 + i_6) \nonumber \\
&& \hspace{0.4in} \times \frac{\prod_{j=1}^6 E\left\{g_{i_j}^2(X_1)\right\} \sigma_{ij}^{i_1} \sigma_{ik}^{i_2}\sigma_{il}^{i_3}\sigma_{jk}^{i_4}\sigma_{jl}^{i_5} \sigma_{kl}^{i_6}}{(i_1 + i_2 + i_3)!(i_1 + i_4 + i_5)!(i_2 + i_5 + i_6)!(i_3 + i_5 + i_6)!} \nonumber \\
&&+ O\left(|\sigma_{ij}|^4 + |\sigma_{ik}|^4 +|\sigma_{il}|^4 + |\sigma_{jk}|^4 + |\sigma_{jl}|^4 + |\sigma_{kl}|^4\right). \label{t-eq-lem6-4}
\end{eqnarray}
where $X_1$ is a standard normal random variable. Recall that \eqref{t-eq-lem6-4} is the first term of $h_{i,j,k,l}(1,1,1,1)$; we next discuss how each subitem under the summation of the main part on the right hand side of \eqref{t-eq-lem6-4} associates with other terms contributes in the evaluation of \eqref{t-eq-lem6-1}. In particular, we shall discuss the possible values of $i_1, \ldots, i_6$; clearly if one of them is zero, its corresponding ``$\sigma$" term in  \eqref{t-eq-lem6-4} does not appear in the expression. 

\begin{itemize}
    \item[(A1)] If the subscripts of the set of $\sigma$'s whose corresponding powers are nonzero contain and only contain three letters, and therefore one letter does not appear in these subscripts, the subitem in \eqref{t-eq-lem6-4} is identical to the corresponding subitem in one and only one of the second to fifth item in $h_{i,j,k,l}(1,1,1,1)$, whose coefficients are all ``$-$". As a consequence, they cancel each other. 
    
    For example $i_3 = i_5 = i_6 = 0$ but at least two of $\{i_1, i_2, i_4\}$ are nonzero, so that the set of $\sigma$ whose corresponding powers are nonzero is from $\{\sigma_{ij}, \sigma_{ik}, \sigma_{jk}\}$, and $\{i,j,k\}$ all appear in the subscript but $l$ does not. The subitem in \eqref{t-eq-lem6-4} is given by 
    \begin{eqnarray*}
    &&C^{i}_{t/2, \hat{\sigma}_i}(i_1 + i_2 )C^{j}_{t/2, \hat{\sigma}_j}(i_1 + i_4 )C^{k}_{t/2, \hat{\sigma}_k}(i_2 + i_4 ) C^{l}_{t/2, \hat{\sigma}_l}(0) \nonumber \\
&& \hspace{0.4in} \times \frac{E\left\{g_{i_1}^2(X_1)\right\}E\left\{g_{i_2}^2(X_1)\right\} E\left\{g_{i_4}^2(X_1)\right\} \sigma_{ij}^{i_1} \sigma_{ik}^{i_2}\sigma_{jk}^{i_4}}{(i_1 + i_2 )!(i_1 + i_4 )!(i_2 + i_4)!},
    \end{eqnarray*}
    which is identical and only identical to a corresponding subitem in the Taylor expansion for the second item in $h_{i,j,k,l}(1,1,1,1)$ by noting that $C^{l}_{t/2, \hat{\sigma}_i}(0) = P^{\ast}(Z_l< \hat{\sigma}_i q_{t/2})$; in particular, similarly to the development of \eqref{t-eq-lem6-4}, we are able to derive:
    \begin{eqnarray}
&&P^{\ast}(Z_i < \hat{\sigma}_i q_{t/2}, Z_j < \hat{\sigma}_j q_{t/2}, Z_k < \hat{\sigma}_k q_{t/2}) \nonumber \\
&=& \sum_{i_1 + i_2 + i_4 \in \{0, 1, 2, 3\} } C^{i}_{t/2, \hat{\sigma}_i}(i_1 + i_2 )C^{j}_{t/2, \hat{\sigma}_j}(i_1 + i_4 )C^{k}_{t/2, \hat{\sigma}_k}(i_2 + i_4 ) \nonumber \\
&& \hspace{0.4in} \times \frac{E\left\{g_{i_1}^2(X_1)\right\}E\left\{g_{i_2}^2(X_1)\right\} E\left\{g_{i_4}^2(X_1)\right\} \sigma_{ij}^{i_1} \sigma_{ik}^{i_2}\sigma_{jk}^{i_4}}{(i_1 + i_2 )!(i_1 + i_4 )!(i_2 + i_4)!} \nonumber \\
&&+ O\left(|\sigma_{ij}|^4 + |\sigma_{ik}|^4 +|\sigma_{jk}|^4\right). \label{t-eq-lem6-5}
\end{eqnarray}
    
    \item[(A2)] If the subscripts of the set $\sigma$'s whose corresponding powers are nonzero contain only two letter in their subscripts, indicating that one and one and only of $i_j, j=1,\ldots, 6$ is nonzero, we can find two of the second to fifth terms and one of the sixth to thirteenth terms in $h_{i,j,k,l}(1,1,1,1)$  contain the same corresponding subitem in \eqref{t-eq-lem6-4}. Note that the coefficients of the second to fifth terms in  $h_{i,j,k,l}(1,1,1,1)$ are ``-", while those of the sixth to thirteenth terms in $h_{i,j,k,l}(1,1,1,1)$ are "$+$", therefore they are cancelled. 
    
    For example, if we have only $i_1>0$, The subitem in \eqref{t-eq-lem6-4} is given by 
    \begin{eqnarray*}
    C^{k}_{t/2, \hat{\sigma}_k}(0) C^{l}_{t/2, \hat{\sigma}_l}(0)  C^{i}_{t/2, \hat{\sigma}_i}(i_1)C^{j}_{t/2, \hat{\sigma}_j}(i_1) \frac{E\left\{g_{i_1}^2(X_1)\right\} \sigma_{ij}^{i_1}}{(i_1!)^2},
    \end{eqnarray*}
    which is identical to the corresponding terms in $P^{\ast}(Z_i< \hat{\sigma}_i q_{t/2}, Z_j<\hat{\sigma}_j q_{t/2}, Z_k<\hat{\sigma}_k q_{t/2})P^{\ast}(Z_l<\hat{\sigma}_lq_{t/2})$ by observing its Taylor expansion given by \eqref{t-eq-lem6-5}, $P^{\ast}(Z_i< \hat{\sigma}_i q_{t/2}, Z_j<\hat{\sigma}_j q_{t/2}, Z_l<\hat{\sigma}_lq_{t/2})P^{\ast}(Z_k<\hat{\sigma}_k q_{t/2})$ (details are omitted), and $P^{\ast}(Z_i< \hat{\sigma}_i q_{t/2}, Z_j<\hat{\sigma}_j q_{t/2})P^{\ast}(Z_k<\hat{\sigma}_k q_{t/2})P^{\ast}(Z_l<\hat{\sigma}_l q_{t/2})$, since
    \begin{eqnarray*}
    &&P^{\ast}(Z_i <\hat{\sigma}_i q_{t/2}, Z_j < \hat{\sigma}_j q_{t/2}) = \sum_{i_1 =0 }^3 \frac{C^{i}_{t/2, \hat{\sigma}_i}(i_1 )C^{j}_{t/2, \hat{\sigma}_j}(i_1 )E\left\{g_{i_1}^2(X_1)\right\} \sigma_{ij}^{i_1} }{(i_1!)^2} + O\left(|\sigma_{ij}|^4\right). \label{t-eq-lem6-6}
    \end{eqnarray*}
    
    \item[(A3)] If $i_j = 0$ for $j=1,\ldots, 6$, it can be checked that all the terms in $h_{i,j,k,l}(1,1,1,1)$ contain the subitem $C^{k}_{t/2, \hat{\sigma}_k}(0) C^{l}_{t/2, \hat{\sigma}_l}(0)  C^{i}_{t/2, \hat{\sigma}_i}(0)C^{j}_{t/2, \hat{\sigma}_j}(0)$, where seven have the coefficient ``$+$", and the other seven have the coefficient ``$-$", therefore they cancel each other. 
    
\end{itemize}

We observe that all the subitems in the second to the sixteenth items in $h_{i,j,k,l}(1,1,1,1)$ have been considered and therefore cancelled in Cases (A1)--(A3). Only the subitems with the set of $\sigma$'s whose corresponding powers are nonzeros contain all four letters in their subscripts. We consider the following possibilities.
\begin{itemize}
    \item[(B1)] Only two $i_j$'s are nonzero; therefore only two $\sigma$'s appear in the subitem in \eqref{t-eq-lem6-4}. The corresponding $\sigma$ components in it have three possibilities: $\sigma_{ij}^{i_1}\sigma_{kl}^{i_6}$, $\sigma_{ik}^{i_2}\sigma_{jl}^{i_5}$, and $\sigma_{il}^{i_3} \sigma_{jk}^{i_4}$. These three cases can be considered similarly; as an example, we consider that it is $\sigma_{ij}^{i_1} \sigma_{kl}^{i_6}$, with $i_1>0$ and $i_6>0$. Furthermore since $0< i_1+i_6 \leq 3$, we must have ``$i_1 = 1, i_6 = 1$" or ``$i_1 = 1, i_6 = 2$" or ``$i_1 = 2, i_6 = 1$". 
    
    \begin{itemize}
        \item If $i_1 = 1$ and $i_6 = 1$, the subitem in \eqref{t-eq-lem6-4} is give by
        \begin{eqnarray*} C^{i}_{t/2, \hat{\sigma}_i}(1)C^{j}_{t/2, \hat{\sigma}_j}(1)C^{k}_{t/2, \hat{\sigma}_k}(1)C^{l}_{t/2, \hat{\sigma}_l}(1)\sigma_{ij}  \sigma_{kl}, \end{eqnarray*}
        and we can check that the corresponding subitem in $h_{i,j,k,l}(a_1, a_2, a_3, a_4)$ is given by 
        \begin{eqnarray*} &&\left\{C^{i}_{t/2, \hat{\sigma}_i}(1)\right\}^{\frac{1+a_1}{2}}\left\{C^{i}_{1-t/2, \hat{\sigma}_i}(1)\right\}^{\frac{1-a_1}{2}}\left\{C^{j}_{t/2, \hat{\sigma}_j}(1)\right\}^{\frac{1+a_2}{2}}\left\{C^{j}_{1-t/2, \hat{\sigma}_j}(1)\right\}^{\frac{1-a_2}{2}}\\
       \times && \left\{C^{k}_{t/2, \hat{\sigma}_k}(1)\right\}^{\frac{1+a_3}{2}}\left\{C^{k}_{1-t/2, \hat{\sigma}_k}(1)\right\}^{\frac{1-a_3}{2}}\left\{C^{l}_{t/2, \hat{\sigma}_l}(1)\right\}^{\frac{1+a_4}{2}}\left\{C^{l}_{1-t/2, \hat{\sigma}_l}(1)\right\}^{\frac{1-a_4}{2}}\sigma_{ij}  \sigma_{kl}.  \end{eqnarray*}
       Using the fact that $C_{t/2, \hat{\sigma}_a}^{a}(1) = C_{1-t/2, \hat{\sigma}_a}^{a}(1)$ when $\mu_a = 0$ given by \eqref{t-eq-lem3-0}, we can check that when one of $\{\mu_i, \mu_j, \mu_k, \mu_l\}$ is equal to 0, all these subtems will cancel each other across $h_{i,j,k,l}(a_1, a_2, a_3, a_4)$ when evaluating \eqref{t-eq-lem6-1}. Therefore this subitem is remained only when  $\{\mu_i, \mu_j, \mu_k, \mu_l\}$ are all nonzero, and their contribution in \eqref{t-eq-lem6-1} is given by 
       \begin{eqnarray*}
       \left\{C_{t/2, \hat{\sigma}_i}^i(1) -  C_{1-t/2, \hat{\sigma}_i}^i(1)\right\}\left\{C_{t/2, \hat{\sigma}_j}^j(1) -  C_{1-t/2, \hat{\sigma}_j}^j(1)\right\}\\\times \left\{C_{t/2, \hat{\sigma}_k}^k(1) -  C_{1-t/2, \hat{\sigma}_k}^k(1)\right\}\left\{C_{t/2, \hat{\sigma}_l}^l(1) -  C_{1-t/2, \hat{\sigma}_l}^l(1)\right\} \sigma_{ij}\sigma_{kl}. 
       \end{eqnarray*}
       
       \item If $i_1=1$, $i_6 =2$, the subitem in \eqref{t-eq-lem6-4} is given by \begin{eqnarray*} \frac{1}{2}C^{i}_{t/2, \hat{\sigma}_i}(1)C^{j}_{t/2, \hat{\sigma}_j}(1)C^{k}_{t/2, \hat{\sigma}_k}(2)C^{l}_{t/2, \hat{\sigma}_l}(2)\sigma_{ij}  \sigma_{kl}^2, \end{eqnarray*} 
       and the corresponding subitem in $h_{i,j,k,l}(a_1, a_2, a_3, a_4)$ is given by 
        \begin{eqnarray*} &&\left\{C^{i}_{t/2, \hat{\sigma}_i}(1)\right\}^{\frac{1+a_1}{2}}\left\{C^{i}_{1-t/2, \hat{\sigma}_i}(1)\right\}^{\frac{1-a_1}{2}}\left\{C^{j}_{t/2, \hat{\sigma}_j}(1)\right\}^{\frac{1+a_2}{2}}\left\{C^{j}_{1-t/2, \hat{\sigma}_j}(1)\right\}^{\frac{1-a_2}{2}}\\
       \times && \left\{C^{k}_{t/2, \hat{\sigma}_k}(2)\right\}^{\frac{1+a_3}{2}}\left\{C^{k}_{1-t/2, \hat{\sigma}_k}(2)\right\}^{\frac{1-a_3}{2}}\left\{C^{l}_{t/2, \hat{\sigma}_l}(2)\right\}^{\frac{1+a_4}{2}}\left\{C^{l}_{1-t/2, \hat{\sigma}_l}(2)\right\}^{\frac{1-a_4}{2}}\sigma_{ij}  \sigma_{kl}^2.  \end{eqnarray*}
       When $\mu_i = 0$ or $\mu_j = 0$, all these subitems will cancel each other across $h_{i,j,k,l}(a_1, a_2, a_3, a_4)$ when evaluating \eqref{t-eq-lem6-1}. Therefore this subitem is remained only when  $\mu_i \neq 0$ and $\mu_j \neq 0$, and their contribution in \eqref{t-eq-lem6-1} is given by 
       \begin{eqnarray*}
       \frac{1}{2}\left\{C_{t/2, \hat{\sigma}_i}^i(1) -  C_{1-t/2, \hat{\sigma}_i}^i(1)\right\}\left\{C_{t/2, \hat{\sigma}_j}^j(1) -  C_{1-t/2, \hat{\sigma}_j}^j(1)\right\}\\\times \left\{C_{t/2, \hat{\sigma}_k}^k(2) -  C_{1-t/2, \hat{\sigma}_k}^k(2)\right\}\left\{C_{t/2, \hat{\sigma}_l}^l(2) -  C_{1-t/2, \hat{\sigma}_l}^l(2)\right\} \sigma_{ij}\sigma_{kl}^2.
       \end{eqnarray*}
       
       \item If $i_1 = 2$ and $i_6 = 1$, we can similarly conclude that the corresponding subitem in  $h_{i,j,k,l}(a_1, a_2, a_3, a_4)$ can be remained only when $\mu_k \neq 0$ and $\mu_l \neq 0$, and their contribution in \eqref{t-eq-lem6-1} is given by
       \begin{eqnarray*}
       \frac{1}{2}\left\{C_{t/2, \hat{\sigma}_i}^i(2) -  C_{1-t/2, \hat{\sigma}_i}^i(2)\right\}\left\{C_{t/2, \hat{\sigma}_j}^j(2) -  C_{1-t/2, \hat{\sigma}_j}^j(2)\right\}\\\times \left\{C_{t/2, \hat{\sigma}_k}^k(1) -  C_{1-t/2, \hat{\sigma}_k}^k(1)\right\}\left\{C_{t/2, \hat{\sigma}_l}^l(1) -  C_{1-t/2, \hat{\sigma}_l}^l(1)\right\} \sigma_{ij}^2\sigma_{kl}.
       \end{eqnarray*}
       
    \end{itemize}
    
    In summary, for subitems considered in (B1), when summed over $i,j,k,l$ for $i\neq j\neq k \neq l$, they are equal to 
    \begin{eqnarray*}
    3\left[\sum_{i\neq j; i,j\in \mathcal{H}_1}\left\{C_{t/2, \hat{\sigma}_i}^i(1) -  C_{1-t/2, \hat{\sigma}_i}^i(1)\right\}\left\{C_{t/2, \hat{\sigma}_j}^j(1) -  C_{1-t/2, \hat{\sigma}_j}^j(1)\right\}\sigma_{ij}\right]^2 \\ + 3\left[\sum_{i\neq j; i,j\in \mathcal{H}_1}\left\{C_{t/2, \hat{\sigma}_i}^i(1) -  C_{1-t/2, \hat{\sigma}_i}^i(1)\right\}\left\{C_{t/2, \hat{\sigma}_j}^j(1) -  C_{1-t/2, \hat{\sigma}_j}^j(1)\right\}\sigma_{ij}\right]\\
    \times \left[\sum_{i\neq j}\left\{C_{t/2, \hat{\sigma}_i}^i(2) -  C_{1-t/2, \hat{\sigma}_i}^i(2)\right\}\left\{C_{t/2, \hat{\sigma}_j}^j(2) -  C_{1-t/2, \hat{\sigma}_j}^j(2)\right\}\sigma_{ij}^2\right]. 
    \end{eqnarray*}

\item[(B2)] There are three $i_j$'s are nonzero, and therefore each is equal to 1. Note that the corresponding subscripts of $\sigma$ have six letters; and each of $i,j,k,l$ must appear at least once. There are two possibilities
\begin{itemize}
    \item One letter appears three times, but each of the other letters appear once. For example, letter $i$ appears three times, but $j,k,l$ appear once, namely $i_1 = i_2 = i_3 = 1$, which corresponds to the subitem in \eqref{t-eq-lem6-4}:
    \begin{eqnarray*}
    C_{t/2, \hat{\sigma}_i}^i(3) C_{t/2, \hat{\sigma}_j}^j(1) C_{t/2, \hat{\sigma}_k}^k(1) C_{t/2, \hat{\sigma}_l}^l(1)\frac{\sigma_{ij} \sigma_{ik}\sigma_{il}}{3!}.
    \end{eqnarray*}
    Using similar arguments as those in (B1), we can conclude that this subitem is remained in $h_{i,j,k,l}(a_1, a_2, a_3, a_4)$ only when $\mu_i, \mu_j, \mu_k, \mu_l$ are all nonzero, and they are in the order of $O(\sigma_{ij}\sigma_{ik}\sigma_{il})$. 
    
    \item Two letters appear twice, but each of the other two letters appear once. For example, letter $i$ and $j$ appear twice, but $k,l$ appear once, namely $i_1 = i_2 = i_5 = 1$, which corresponds to the subitem in \eqref{t-eq-lem6-4}:
    \begin{eqnarray*}
    C_{t/2, \hat{\sigma}_i}^i(2) C_{t/2, \hat{\sigma}_j}^j(2) C_{t/2, \hat{\sigma}_k}^k(1) C_{t/2, \hat{\sigma}_l}^l(1)\frac{\sigma_{ij} \sigma_{ik}\sigma_{jl}}{2!2!}.
    \end{eqnarray*}
    Using similar arguments as those in (B1), we can conclude that this subitem is remained in $h_{i,j,k,l}(a_1, a_2, a_3, a_4)$ only when $\mu_k \neq 0$ and $\mu_l \neq 0$, and they are in the order of $O(\sigma_{ij} \sigma_{ik}\sigma_{jl})$. 
    
\end{itemize}

In summary, the subitems considered in (B2), when summed over $i,j,k,l$ for $i\neq j\neq k\neq l$, they are in the order of 
\begin{eqnarray*}
O\left( \sum_{i\in \mathcal{H}_1}\left( \sum_{j\neq i, j\in\mathcal{H}_1} |\sigma_{ij}|\right)^3 \right) + O\left( \sum_{i\neq j} |\sigma_{ij}| \sum_{k\neq i, k\in \mathcal{H}_1} |\sigma_{ik}|\sum_{l\neq j, l\in \mathcal{H}_1}|\sigma_{jl}| \right). 
\end{eqnarray*}
    
\end{itemize}

\end{proof}

\begin{lemma}\label{t-lem-5}
Recall the definition: $t_i = \mathrm{1}(|T_i| > |q_{t/2}|) = \mathrm{1}(P_i < t)$; we have
\begin{equation*}
    \begin{split}
        &\sum_{i\neq j\neq k}\Big|E(t_i t_j t_k) - E(t_i t_j)E(t_k) - E(t_i t_k)E(t_j) - E(t_j t_k)E(t_i) + 2E(t_i)E(t_j)E(t_k)\Big|\\
        = & O\left(\sum_{i=1}^p \left(\sum_{j\neq i, j\in \mathcal{H}_1} |\sigma_{ij}|\right)^2 \right) + O\left(\sum_{i\in \mathcal{H}_1}\left( \sum_{j\neq i} |\sigma_{ij}|\right) \left( \sum_{k\neq i, k\in \mathcal{H}_1}|\sigma_{ik}|\right) \right) \\ &+  O\left(\sum_{i\neq j\neq k} |\sigma_{ij}\sigma_{ik} \sigma_{kl}| \right) + O\left (p \sum_{i\neq j} \sigma_{ij}^4 \right). \\
    \end{split}
\end{equation*}
\end{lemma}

\begin{proof}
The proof of this lemma is similar to but simpler than that of Lemma \ref{t-lem-6}. The details are thus omitted. 
\end{proof}

\begin{lemma}\label{t-lem-7}
Let
\begin{eqnarray*}
m(\bar{V},\bar{R}) = \frac{\bar{V}}{E(\bar{R})} - \frac{E(\bar{V})}{[E(\bar{R})]^2} \bar{R}. 
\end{eqnarray*}
Assume $\limsup_{p\to \infty} p_0t/(p_1\bar \xi) < 1$, and as $p$ is sufficiently large,  
\begin{eqnarray}
&& \sum_{i,j}|\sigma_{ij}| = O\left(p^{2-\delta}\right), \label{t-eq-lem7-assum-0}\\
&&\sum_{i\neq j; i,j\in \mathcal{H}_0} \sigma_{ij}^2 \geq \frac{C_t^{\max}}{E^2(\phi(\hat{\sigma}_i q_{t/2})|\hat{\sigma}_iq_{t/2}|)} \sum_{i\in \mathcal{H}_1, j\in \mathcal{H}_0, \mu_i\in[-\mu_t, \mu_t]} \sigma_{ij}^2, \label{t-eq-lem7-assum-1-added} \\
&&\sum_{i\neq j; i,j\in \mathcal{H}_0} \sigma_{ij}^2 + p \gtrsim \sum_{i\neq j; i, j\in \mathcal{H}_1} \sigma_{ij}^2  \label{t-eq-lem7-assum-1-added-2}\\
&& \sum_{i\neq j} \sigma_{ij}^4 = o\left(\sum_{i\neq j; i,j\in \mathcal{H}_0} \sigma_{ij}^2  + p_0\right),  \label{t-eq-lem7-assum-3}
\end{eqnarray}
where $\mu_t$ and $C_t^{\max}$ are defined in Lemma \ref{t-lem-4}. 
We have
\begin{eqnarray}
E\left[\{R-E(R)\}^4\right] &=& o\left(p^4\mbox{Var}(m(\bar V, \bar R)\right) \label{t-eq-lem7-0-1}\\
E\left[\{V-E(V)\}^4\right] &=& o\left(p^4\mbox{Var}(m(\bar V, \bar R)\right).  \label{t-eq-lem7-0-2}
\end{eqnarray}
\end{lemma}

\begin{proof} 
We consider $\mbox{Var}\left(m(\bar{V},\bar{R})\right) $ first. Recall that $R = V + S$; therefore $\bar R = \bar V + \bar S$ with $\bar R = R/p$, $\bar V = V/p$ and $\bar S = S/p$, and 
\begin{eqnarray*}
m(\bar{V},\bar{R}) &=& \frac{E(\bar S)}{\{E(\bar R)\}^2} \bar V - \frac{E(\bar V)}{\{E(\bar R)\}^2} \bar S\\
&=& A \left( \frac{\bar V}{E(\bar V)} - \frac{\bar S}{E(\bar S)} \right),
\end{eqnarray*}
where $A = E(\bar S) E(\bar V)/\{E(\bar R)\}^2$. We shall derive
\begin{eqnarray}
\mbox{Var}\left\{m(\bar{V},\bar{R})\right\} \gtrsim \frac{\sum_{i\neq j; i,j\in \mathcal{H}_0} \sigma_{ij}^2 + p_0  + \mbox{Var}(S)}{p^2}.  \label{t-eq-lem7-1}
\end{eqnarray}
To this end, 
consider 
\begin{eqnarray}
\frac{1}{A^2}\mbox{Var}\left(m(\bar{V},\bar{R})\right) &=& \mbox{Var}(\bar V/E(\bar V)) +  \mbox{Var}(\bar S/E(\bar S)) - 2\mbox{Cov}(\bar V/E(\bar V), \bar S/E(\bar S)), \label{t-eq-lem7-1-added-6}
\end{eqnarray}
and based on Lemma \ref{t-lem-4},
\begin{eqnarray}
&&\mbox{Cov}(\bar V/E(\bar V), \bar S/E(\bar S)) = \frac{1}{E(V) E(S)} \sum_{i\in \mathcal{H}_1, j\in \mathcal{H}_0} \left\{ E(t_it_j) - E(t_i)E(t_j) \right\} \nonumber\\
&\leq& \frac{\phi( C_t^{\max}}{p_0 p_1 t \bar \xi}  \sum_{i\in \mathcal{H}_1, j\in \mathcal{H}_0; \mu_i \in[-\mu_t, \mu_t]} \sigma_{ij}^2 + \frac{1}{p_0 p_1 t \bar \xi} O\left(\sum_{i\in \mathcal{H}_1, j\in \mathcal{H}_0} \sigma_{ij}^4\right).  \label{t-eq-lem7-1-added-7}
\end{eqnarray}
Combining \eqref{t-eq-lem7-1-added-6} with \eqref{t-eq-lem7-1-added-7}, and applying Lemma \ref{t-lem-4}, we have
\begin{eqnarray*}
\frac{1}{A^2}\mbox{Var}\left(m(\bar{V},\bar{R})\right) &\geq& \frac{2E^2(\phi(\hat{\sigma}_i|q_{t/2}|)\hat{\sigma}_iq_{t/2}) \sum_{i\neq j; i,j\in \mathcal{H}_0} \sigma_{ij}^2 + p_0t(1-t) + O\left( \sum_{i\neq j, i,j\in \mathcal{H}_0} \sigma_{ij}^4\right)}{p_0^2 t^2}\\
&&-\frac{p_0 t}{p_1 \bar \xi}\frac{2E^2(\phi(\hat{\sigma}_iq_{t/2})\hat{\sigma}_i|q_{t/2}|)C_t^{\max}}{p_0^2  t^2}  \sum_{i\in \mathcal{H}_1, j\in \mathcal{H}_0; \mu_i \in [-\mu_t, \mu_t]} \sigma_{ij}^2 - \frac{O\left(\sum_{i\in \mathcal{H}_1, j\in \mathcal{H}_0} \sigma_{ij}^4\right)}{p_0p_1t\bar \xi} \\&&+ \mbox{Var}(S/E(S))\\
&\gtrsim& \frac{\sum_{i\neq j; i,j\in \mathcal{H}_0} \sigma_{ij}^2 + p_0}{p_0^2} + \mbox{Var}(S/E(S)),
\end{eqnarray*}
where  ``$\gtrsim$" is based on assumptions \eqref{t-eq-lem7-assum-1-added} and \eqref{t-eq-lem7-assum-3}; this verifies \eqref{t-eq-lem7-1}.

We next consider $E[\{R-E(R)\}^4]$:
\begin{equation}\label{t-eq-lem7-2}
    \begin{split}
         &E[\{R-E(R)\}^4] =  E(R^4) - 4E(R^3)E(R) + 6E(R^2)E^2(R) - 3E^4(R)\\
        = & E\left\{\left(\sum_{i=1}^p t_i\right)^4\right\} - 4E\left\{\left(\sum_{i=1}^p t_i\right)^3\right\}E\left(\sum_{i=1}^p t_i\right) \\ &+ 6E\left\{\left(\sum_{i=1}^p t_i\right)^2\right\}E^2\left(\sum_{i=1}^p t_i\right) - 3E^4\left(\sum_{i=1}^p t_i\right)\\
        = & E\left(\sum_{i, j, k, l = 1}^p t_i t_j t_k t_l\right) - 4E\left(\sum_{i, j, k = 1}^pt_it_jt_k\right)E\left(\sum_{l=1}^pt_l\right) \\ &+ 6E\left(\sum_{i, j = 1}^p t_it_j\right)E\left(\sum_{k=1}^pt_k\right)E\left(\sum_{l=1}^pt_l\right)\\
        & - 3E\left(\sum_{i= 1}^p t_i\right)E\left(\sum_{j= 1}^p t_j\right)E\left(\sum_{k= 1}^p t_k\right)E\left(\sum_{l= 1}^p t_l\right). 
    \end{split}
\end{equation}
We need to consider the following possibilities:
\begin{itemize}
    \item[(C1)] Based on Lemma \ref{t-lem-5}, the collection of terms in (\ref{t-eq-lem7-2}) where $\{i, j, k, l\}$ are all different is given by 
    \begin{equation}
        \begin{split}
            & \sum_{i\neq j\neq k \neq l}\Big\{E(t_i t_j t_k t_l) - E(t_i t_j t_k)E(t_l) - E(t_i t_j t_l)E(t_k) - E(t_i t_k t_l)E(t_j) - E(t_j t_k t_l)E(t_i)\\
        + &E(t_i t_j)E(t_k)E(t_l) + E(t_i t_k)E(t_j)E(t_l) + E(t_i t_l)E(t_j)E(t_k) + E(t_j t_k)E(t_i)E(t_l)\\
        + &E(t_j t_l)E(t_i)E(t_k) + E(t_k t_l)E(t_i)E(t_j) -3E(t_i)E(t_j)E(t_k)E(t_l)\Big\}\\
        = &3 E\Big\{\left[\sum_{i\neq j; i,j\in \mathcal{H}_1}\left\{C_{t/2, \hat{\sigma}_i}^i(1) -  C_{1-t/2, \hat{\sigma}_i}^i(1)\right\}\left\{C_{t/2, \hat{\sigma}_j}^j(1) -  C_{1-t/2, \hat{\sigma}_j}^j(1)\right\}\sigma_{ij}\right]^2\Big\} \\  &+ 3E\Big\{\left[\sum_{i\neq j; i,j\in \mathcal{H}_1}\left\{C_{t/2, \hat{\sigma}_i}^i(1) -  C_{1-t/2, \hat{\sigma}_i}^i(1)\right\}\left\{C_{t/2, \hat{\sigma}_j}^j(1) -  C_{1-t/2, \hat{\sigma}_j}^j(1)\right\}\sigma_{ij}\right]\\
    &\times \left[\sum_{i\neq j}\left\{C_{t/2, \hat{\sigma}_i}^i(2) -  C_{1-t/2, \hat{\sigma}_i}^i(2)\right\}\left\{C_{t/2, \hat{\sigma}_j}^j(2) -  C_{1-t/2, \hat{\sigma}_j}^j(2)\right\}\sigma_{ij}^2\right]\Big\} \\
        & + O\left( \sum_{i\in \mathcal{H}_1}\left( \sum_{j\neq i, j\in\mathcal{H}_1} |\sigma_{ij}|\right)^3 \right) + O\left( \sum_{i\neq j} |\sigma_{ij}| \sum_{k\neq i, k\in \mathcal{H}_1} |\sigma_{ik}|\sum_{l\neq j, l\in \mathcal{H}_1}|\sigma_{jl}| \right)+ O\left( p^2 \sum_{i\neq j}\sigma_{ij}^4\right)\\
        &\equiv 3E(\mathcal{I}_1) + 3E(\mathcal{I}_2) + \mathcal{I}_3 + \mathcal{I}_4 + \mathcal{I}_5. 
        \end{split} \label{t-eq-lem7-3}
    \end{equation}
    We consider the above $\mathcal{I}$ terms one by one.  We consider $\mathcal{I}_1$ and $\mathcal{I}_2$ first. Note that 
    \begin{eqnarray*}
    \mathcal{I}_1 = \mathcal{I}_{1,1} \cdot \mathcal{I}_{1,1} \\
    \mathcal{I}_2 = \mathcal{I}_{1,1} \cdot \mathcal{I}_{2,2},
    \end{eqnarray*}
    with
    \begin{eqnarray*}
    \mathcal{I}_{1,1} &=& \sum_{i\neq j; i,j\in \mathcal{H}_1}\left\{C_{t/2, \hat{\sigma}_i}^i(1) -  C_{1-t/2, \hat{\sigma}_i}^i(1)\right\}\left\{C_{t/2, \hat{\sigma}_j}^j(1) -  C_{1-t/2, \hat{\sigma}_j}^j(1)\right\}\sigma_{ij}\\
    \mathcal{I}_{2,2} &=& \sum_{i\neq j}\left\{C_{t/2, \hat{\sigma}_i}^i(2) -  C_{1-t/2, \hat{\sigma}_i}^i(2)\right\}\left\{C_{t/2, \hat{\sigma}_j}^j(2) -  C_{1-t/2, \hat{\sigma}_j}^j(2)\right\}\sigma_{ij}^2.
    \end{eqnarray*}
    As a consequence
    \begin{eqnarray}
    \mathcal{I}_1 + \mathcal{I}_2 = \mathcal{I}_{1,1}(\mathcal{I}_{1,1} + \mathcal{I}_{2,2}), \label{t-eq-lem7-4-added}
    \end{eqnarray}
    and 
    we need to verify
    \begin{eqnarray}
        E(\mathcal{I}_1 + \mathcal{I}_2)  = o\left(p^4\mbox{Var}(\bar V, \bar R) \right) \label{t-eq-lem7-4}
        \end{eqnarray}
    Based on the assumption \eqref{t-eq-lem7-assum-0}, we immediately have 
    \begin{eqnarray}
    \mathcal{I}_{1,1} = o(p^2), \label{t-eq-lem7-5}
    \end{eqnarray}
    and we have the following decomposition for $\mathcal{I}_{2,2} + \mathcal{I}_{1,1}$:
    \begin{eqnarray}
    E(\mathcal{I}_{2,2} + \mathcal{I}_{1,1}) &=&  \left\{E(\mathcal{I}_{1,1} +\frac{1}{2}\mathcal{I}_{2,2} + \mathcal{I}_{2,3}) + p_0t(1-t) +  \sum_{i=1}^{p_1} \xi_i (1-\xi_i)\right\} \nonumber \\&& + \left\{E(\frac{1}{2}\mathcal{I}_{2,2} -\mathcal{I}_{2,3}) - p_0t(1-t) -  \sum_{i=1}^{p_1} \xi_i (1-\xi_i)\right\}, \label{t-eq-lem7-6}
    \end{eqnarray}
    where 
    \begin{eqnarray*}
    \mathcal{I}_{2,3} = \frac{E\{g_3^2(X_1)\}}{(3!)^2}\sum_{i \neq j} \left\{C_{t/2, \sigma_i}^i(3) -  C_{1-t/2, \sigma_i}^i(3)\right\}\left\{C_{t/2, \sigma_j}^j(3) -  C_{1-t/2, \sigma_j}^j(3)\right\}\sigma_{ij}^3.  \label{t-eq-lem7-7}
    \end{eqnarray*}
    Furthermore, with Lemmas \ref{t-lem-3} and \ref{t-lem-4}, we are able to check 
    \begin{eqnarray}
    \mbox{Var}(R) &=& \sum_{i\neq j} \{E(t_it_j) - E(t_i)E(t_j)\} + \sum_{i=1}^p \left[E(t_i^2) - \{E(t_i)\}^2\right] \nonumber \\
    &=& E(\mathcal{I}_{1,1} +\frac{1}{2}\mathcal{I}_{2,2} + \mathcal{I}_{2,3}) + p_0t(1-t) +  \sum_{i=1}^{p_1} \xi_i (1-\xi_i) + O\left(\sum_{i\neq j} \sigma_{ij}^4 \right), \label{t-eq-lem7-8}
    \end{eqnarray}
    and
    \begin{eqnarray}
    &&\mbox{Var}(V) + 2 \mbox{Cov}(V, S) \nonumber \\&=& \sum_{i\neq j; i,j\in \mathcal{H}_0} \{E(t_it_j) - E(t_i)E(t_j)\} \nonumber \\&& + \sum_{i\in \mathcal{H}_1, j\in \mathcal{H}_0} \{E(t_it_j) - E(t_i)E(t_j)\} + \sum_{i=1}^{p_0} \left[E(t_i^2) - \{E(t_i)\}^2\right]\nonumber \\
    &=& E(\frac{1}{2}\mathcal{I}_{2,2}) + p_0t(1-t) \nonumber \\ && - E(\frac{1}{2}\sum_{i\neq j; i,j\in \mathcal{H}_1} \left\{C_{t/2, \sigma_i}^i(2) -  C_{1-t/2, \sigma_i}^i(2)\right\}\left\{C_{t/2, \sigma_j}^j(2) -  C_{1-t/2, \sigma_j}^j(2)\right\}\sigma_{ij}^2) \nonumber \\
    && + O\left(\sum_{i\neq j}\sigma_{ij}^4 \right)\label{t-eq-lem7-9}
    \end{eqnarray}
    Combining \eqref{t-eq-lem7-6}, \eqref{t-eq-lem7-8}, and  \eqref{t-eq-lem7-9} leads to 
    \begin{eqnarray}
    E(\mathcal{I}_{2,2} + \mathcal{I}_{1,1}) &=& \mbox{Var}(R) + \mbox{Var}(V) + 2 \mbox{Cov}(V,S) - 2p_0t(1-t) - \sum_{i=1}^{p_1} \xi_i (1-\xi_i)\nonumber\\&&
    - E(I_{2,3}) + E(\frac{1}{2}\sum_{i\neq j; i,j\in \mathcal{H}_1} \left\{C_{t/2, \sigma_i}^i(2) -  C_{1-t/2, \sigma_i}^i(2)\right\}\left\{C_{t/2, \sigma_j}^j(2) -  C_{1-t/2, \sigma_j}^j(2)\right\}\sigma_{ij}^2)\nonumber\\
    && + O\left(\sum_{i\neq j}\sigma_{ij}^4 \right)\nonumber\\
    &=& 2\mbox{Var}(V) + 4\mbox{Cov}(V, S) +  \mbox{Var}(S) - 2p_0t(1-t) - \sum_{i=1}^{p_1} \xi_i (1-\xi_i)\nonumber\\&&
    - I_{2,3} + E(\frac{1}{2}\sum_{i\neq j; i,j\in \mathcal{H}_1} \left\{C_{t/2, \sigma_i}^i(2) -  C_{1-t/2, \sigma_i}^i(2)\right\}\left\{C_{t/2, \sigma_j}^j(2) -  C_{1-t/2, \sigma_j}^j(2)\right\}\sigma_{ij}^2)\nonumber \\
    && + O\left(\sum_{i\neq j}\sigma_{ij}^4 \right). \label{t-eq-lem7-10}
    \end{eqnarray}
    Based on \eqref{t-eq-lem7-1} and Cauchy-Schwartz inequality, we conclude 
    \begin{equation}
    2\mbox{Var}(V) + 4\mbox{Cov}(V, S) +  \mbox{Var}(S) - 2p_0t(1-t) - \sum_{i=1}^{p_1} \xi_i (1-\xi_i) = O\left(p^2\mbox{Var}(m(\bar V, \bar R)\right); \label{t-eq-lem7-11}
    \end{equation}
    and based on the assumptions \eqref{t-eq-lem7-assum-1-added-2} and \eqref{t-eq-lem7-assum-3}, we have
    \begin{equation}
    E(\frac{1}{2} \sum_{i\neq j; i,j\in \mathcal{H}_1} \left\{C_{t/2, \sigma_i}^i(2) -  C_{1-t/2, \sigma_i}^i(2)\right\}\left\{C_{t/2, \sigma_j}^j(2) -  C_{1-t/2, \sigma_j}^j(2)\right\})\sigma_{ij}^2
    + O\left(\sum_{i\neq j}\sigma_{ij}^4 \right) = O\left(p^2\mbox{Var}(m(\bar V, \bar R)\right) \label{t-eq-lem7-12}
    \end{equation}
    Combining \eqref{t-eq-lem7-10} \eqref{t-eq-lem7-11} and \eqref{t-eq-lem7-12}, we have
    \begin{eqnarray*}
     E(\mathcal{I}_{2,2} + \mathcal{I}_{1,1}) &=& E(-\mathcal{I}_{2,3}) + O\left(p^2\mbox{Var}(m(\bar V, \bar R)\right),
    \end{eqnarray*}
    which together with \eqref{t-eq-lem7-5} leads to 
    \begin{eqnarray}
    E(\mathcal{I}_{1,1}(\mathcal{I}_{2,2} + \mathcal{I}_{1,1})) = -o(p^2E(\mathcal{I}_{2,3})) + o\left(p^4\mbox{Var}(m(\bar V, \bar R)\right). \label{t-eq-lem7-13}
    \end{eqnarray}
    For $\mathcal{I}_{1,1}\mathcal{I}_{2,3}$:
    \begin{eqnarray*}
    |\mathcal{I}_{1,1}\mathcal{I}_{2,3}| \lesssim o(p^2\sum_{i\neq j}|\sigma_{i,j}|^3) = o\left(p^4\mbox{Var}(m(\bar V, \bar R)\right),
    \end{eqnarray*}
    which together with \eqref{t-eq-lem7-13} and \eqref{t-eq-lem7-4-added} verifies \eqref{t-eq-lem7-4}. 

  We proceed to consider $\mathcal{I}_3$. 
  \begin{eqnarray}
  \sum_{i\neq j\neq k\neq l; i,j,k,l\in \mathcal{H}_1}|\sigma_{ij}||\sigma_{ik}||\sigma_{il}|&\leq& \sum_{k\ne l, k,l\in \mathcal{H}_1}\left(\sum_{i\neq j, i,j\in \mathcal{H}_1} \sigma_{ij}^2\right)^{1/2}\left(\sum_{i\neq j\neq k\neq l; i,j\in \mathcal{H}_1}\sigma_{ik}^2 \sigma_{il}^2\right)^{1/2}\nonumber\\
  &=& \left(\sum_{i\neq j; i,j\in \mathcal{H}_1} \sigma_{ij}^2\right)^{1/2} \sum_{k\neq l; k,l\in \mathcal{H}_1}\left\{(p_1-3) \sum_{i\neq k\neq l; i\in \mathcal{H}_1}(\sigma_{ik}^2 \sigma_{il}^2)\right\}^{1/2}\nonumber \\
  &\leq& p_1(p_1-3)^{1/2}\left(\sum_{i\neq j; i,j\in \mathcal{H}_1} \sigma_{ij}^2\right)^{1/2} \left\{ \sum_{k\neq l\neq i; i,k,l\in \mathcal{H}_1} (\sigma_{ik}^2 \sigma_{il}^2) \right\}^{1/2}\nonumber \\
  &\lesssim & p_1^{3/2}\left(\sum_{i\neq j; i,j\in \mathcal{H}_1} \sigma_{ij}^2\right)^{1/2} \left( p_1\sum_{i\neq k; i,k\in \mathcal{H}_1} \sigma_{ik}^4 + p\sum_{i\neq l; i,l\in \mathcal{H}_1} \sigma_{il}^4 \right)^{1/2}\nonumber\\
  &=& p_1^2 \left(\sum_{i\neq j; i,j\in \mathcal{H}_1} \sigma_{ij}^2\right)^{1/2} \left(2\sum_{i\neq j; i,j\in \mathcal{H}_1}\sigma_{ij}^4\right)^{1/2}, \label{t-eq-lem7-14-added-1}
  \end{eqnarray}
  where the first ``$\leq$" is because of the Cauchy–Schwarz inequality, the second ``$\leq$" is derived from the Jensen's inequality by noting that $\sqrt{x}$ is a concave function, the ``$\lesssim$" is based on $\sigma_{ik}^2\sigma_{il}^2 \lesssim \sigma_{ik}^4 + \sigma_{il}^4$. The far right hand side of \eqref{t-eq-lem7-14-added-1}
  is in the order of $o\left(p^4\mbox{Var}(m(\bar V, \bar R)\right)$ based on \eqref{t-eq-lem7-1} and the assumptions \eqref{t-eq-lem7-assum-1-added-2} and \eqref{t-eq-lem7-assum-3}; this implies  
  \begin{eqnarray}
  \mathcal{I}_3 = O\left( \sum_{i\in \mathcal{H}_1}\left( \sum_{j\neq i, j\in\mathcal{H}_1} |\sigma_{ij}|\right)^3 \right) = o\left(p^4\mbox{Var}(m(\bar V, \bar R)\right). \label{t-eq-lem7-14}
  \end{eqnarray}
 Similarly to the development of \eqref{t-eq-lem7-14-added-1}, we can obtain
  \begin{eqnarray}
  \mathcal{I}_4 = O\left( \sum_{i\neq j} |\sigma_{ij}| \sum_{k\neq i, k\in \mathcal{H}_1} |\sigma_{ik}|\sum_{l\neq j, l\in \mathcal{H}_1}|\sigma_{jl}| \right) = o\left(p^4\mbox{Var}(m(\bar V, \bar R)\right).  \label{t-eq-lem7-15}
  \end{eqnarray}
  Last, 
  \begin{eqnarray}
  \mathcal{I}_5 = O\left( p^2 \sum_{i\neq j}\sigma_{ij}^4\right)  =  o\left(p^4\mbox{Var}(m(\bar V, \bar R)\right) \label{t-eq-lem7-16}
  \end{eqnarray}
  is implied by \eqref{t-eq-lem7-1} and the assumption \eqref{t-eq-lem7-assum-3}. 
  
  Combining \eqref{t-eq-lem7-4}, \eqref{t-eq-lem7-14}, \eqref{t-eq-lem7-15}, and \eqref{t-eq-lem7-16},  we conclude that the collection of terms in \eqref{t-eq-lem7-2} where $\{i,j,k,l\}$ are all different is in the order of $o\left(p^4\mbox{Var}(m(\bar V, \bar R)\right)$. 
  
    \item[(C2)] If one and only one pair in $\{i, j, k, l\}$ are equal to each other, as an example, we consider $i\neq j \neq k=l$, the collection of terms in (\ref{t-eq-lem7-2}) is given by
    \begin{equation}
        \begin{split}
            & \sum_{i\neq j\neq k}\Big\{E(t_it_j t_k) - E(t_it_k)E(t_j) - E(t_jt_k)E(t_i) - 2E(t_i t_j t_k)E(t_k)
        + E(t_i)E(t_j)E(t_k) \\& + 2E(t_i t_k)E(t_j)E(t_k) + 2E(t_j t_k)E(t_i)E(t_k) + E(t_i t_j)E^2(t_k) -3E^2(t_k)E(t_i)E(t_j)\Big\}\\
        = & \sum_{i\neq j \neq k} \Big[\big\{E(t_i t_j t_k) - E(t_i t_k)E(t_j) - E(t_j t_k)E(t_i) - E(t_i t_j)E(t_k) \\ & \hspace{0.6in} + 2E(t_i)E(t_j)E(t_k)\big\}\{1 - 2E(t_k)\}\Big] \\ &+ \sum_{i\neq j \neq k}\big\{E(t_k) - E^2(t_k)\big\}\big\{E(t_i t_j) - E(t_i)E(t_j)\big\}\\
        =& O\left(\sum_{i=1}^p \left(\sum_{j\neq i, j\in \mathcal{H}_1} |\sigma_{ij}|\right)^2 \right) + O\left(\sum_{i\in \mathcal{H}_1}\left( \sum_{j\neq i} |\sigma_{ij}|\right) \left( \sum_{k\neq i, k\in \mathcal{H}_1}|\sigma_{ik}|\right) \right) \\ &+  O\left(\sum_{i\neq j\neq k} |\sigma_{ij}\sigma_{ik} \sigma_{kl}| \right) + O\left (p \sum_{i\neq j} \sigma_{ij}^4 \right) + O\left(p \sum_{i\neq j}\sigma_{ij}^2\right) + O\left(p \sum_{i\neq j; i,j \in \mathcal{H}_1} |\sigma_{ij}|\right)\\
        =& \mathcal{I}_6 + \mathcal{I}_7 + \mathcal{I}_8 + \mathcal{I}_9 + \mathcal{I}_{10} + \mathcal{I}_{11},
        \end{split} \label{t-eq-lem7-17}
    \end{equation}
    where to achieve the second ``$=$", we have applied Lemmas \ref{t-lem-4} and \ref{t-lem-5}. Next, we consider these ``$\mathcal{I}$" terms one by one. 
    For $\mathcal{I}_6$, since $|\sigma_{ij}|<1$, therefore by noting \eqref{t-eq-lem7-1} and the assumption \eqref{t-eq-lem7-assum-0}, we have
    \begin{eqnarray}
    \mathcal{I}_6 = O\left(\sum_{i=1}^p \left(\sum_{j\neq i, j\in \mathcal{H}_1} |\sigma_{ij}|\right)^2 \right)\lesssim p_1 \sum_{j\neq i, j\in \mathcal{H}_1}|\sigma_{ij}| = p_1O(p^{2-\delta})= o\left(p^4\mbox{Var}(m(\bar V, \bar R)\right);
    \end{eqnarray}
    and similarly
    \begin{eqnarray}
    \mathcal{I}_7 = O\left(\sum_{i\in \mathcal{H}_1}\left( \sum_{j\neq i} |\sigma_{ij}|\right) \left( \sum_{k\neq i, k\in \mathcal{H}_1}|\sigma_{ik}|\right) \right) \lesssim p \sum_{i\neq k; i,k \in \mathcal{H}_1} |\sigma_{ik}| = o\left(p^4\mbox{Var}(m(\bar V, \bar R)\right). 
    \end{eqnarray}
    For $\mathcal{I}_8$, since $|\sigma_{ij}|<1$, we have
    \begin{eqnarray}
    \mathcal{I}_8 = O\left(\sum_{i\neq j\neq k} |\sigma_{ij}\sigma_{ik} \sigma_{kl}| \right) \lesssim p \sum_{i\neq j} |\sigma_{ij}| = o\left(p^4\mbox{Var}(m(\bar V, \bar R)\right).
    \end{eqnarray}
    For $\mathcal{I}_9$ and $\mathcal{I}_{10}$, clearly
    \begin{eqnarray}
    \mathcal{I}_9 \lesssim \mathcal{I}_{10} = o\left(p^4\mbox{Var}(m(\bar V, \bar R))\right);
    \end{eqnarray}
    and 
    \begin{eqnarray}
    \mathcal{I}_{11} = O\left(p \sum_{i\neq j; i,j \in \mathcal{H}_1} |\sigma_{ij}|\right) = o\left(p^4\mbox{Var}(m(\bar V, \bar R))\right) \label{t-eq-lem7-18}
    \end{eqnarray}
    is obtained based on the assumption \eqref{t-eq-lem7-assum-0}. 
    
    Combining \eqref{t-eq-lem7-17}--\eqref{t-eq-lem7-18}, we conclude that the collection of terms in \eqref{t-eq-lem7-2} where one and only one pair in  $\{i,j,k,l\}$ are equal to each other is in the order of $o\left(p^4\mbox{Var}(m(\bar V, \bar R)\right)$.
    
    \item[(C3)] The collection of terms in \eqref{t-eq-lem7-2} where two pairs in 
    $\{i, j, k, l\}$ are equal to each other carries only $O(p^2)$ terms; then based on \eqref{t-eq-lem7-1}, we conclude that this collection of terms is in the order of $o\left(p^4\mbox{Var}(m(\bar V, \bar R)\right)$. 
    
    \item[(C4)] The collection of terms in \eqref{t-eq-lem7-2} where at least three of $\{i, j, k, l\}$ are mutually equal to each other also carries $O(p^2)$ terms, therefore this collection of terms is  in the order of $o\left(p^4\mbox{Var}(m(\bar V, \bar R)\right)$.  

\end{itemize}

In summary, the discussion of (C1)--(C4) implies 
\begin{eqnarray*}
E[\{R-E(R)\}^4] = o\left(p^4\mbox{Var}(m(\bar V, \bar R))\right),
\end{eqnarray*}
which complete the proof for \eqref{t-eq-lem7-0-1}. 

The proof for \eqref{t-eq-lem7-0-2} is similar but simpler. In particular, similar developments in Lemmas \ref{t-lem-6} and \ref{t-lem-7} can be applied to derive the rates of convergence for the terms in $E(V-E(V))^4$.  
But we observe that
since $V$ is the summation of $t_i$ over $i\in \mathcal{H}_0$, the corresponding terms involving $\mathcal{H}_1$ in both Lemmas will not appear, which  simplifies the developments. Here we skip the details for this development to avoid lengthy presentations.  
\end{proof}

\section{Proof of Theorem \ref{Thm-4}}\label{sec-proof-Thm-4}

According to \ref{t-lem-4}, 
\begin{eqnarray*}
\Var(p^{-1}\sum_{i=1}^p t_i) = \dfrac{O(\sum_{i, j = 1}^p \sigma_{ij})}{p^2} = O(p^{-\delta}).
\end{eqnarray*}
Based on Strong Law of Large Numbers for Weakly Correlated Variables (\citet{lyons1988strong}), we have
\begin{eqnarray*}
p^{-1}\sum_{i = 1}^p\{t_i - \xi_i\}\convergeP0\text{ a.s.}
\end{eqnarray*}
Similarly, we can show
\begin{eqnarray*}
p_0^{-1}\sum_{i \in \mathcal{H}_0}\{t_i - \xi_i\}\convergeP0\text{ a.s.}
\end{eqnarray*}
Therefore, based on proof of Theorem 1 in \citet{fan2012estimating}, 
\begin{eqnarray*}
\lim_{p\to\infty} \left[\mathrm{FDP}(t) -\frac{\sum_{i \in \mathcal{H}_0}\xi_i}{\sum_{i=1}^{p}\xi_i}\right] = 0.
\end{eqnarray*}

\section{Proof of Theorem \ref{Thm-5} and Corollary \ref{Corrollary-asym-var-unknown}} \label{sec-proof-Thm-5}

Define $\bar{V} = V(t)/p$ and $\bar{R} = R(t)/p$, then $\FDP(t) = \bar{V}/\bar{R}$. Let $H(v,r)= v/r$ be a function of $(v,r)$ and apply Taylor expansion on this function at $(E(\bar V), E(\bar R))$ and plugging in $v = \bar V$, $r = \bar R \vee c$ to the expansion, where $\bar R\vee c = \max\{\bar R, c\}$ and $c$ is a sufficiently small constant with $0<c<0.5E(\bar R)$; we have
\begin{eqnarray*}
\frac{\bar V}{\bar R\vee c} = \frac{E(\bar{V})}{E(\bar{R})} + \frac{\bar{V} - E(\bar{V})}{E(\bar{R})} - \frac{E(\bar{V})}{\{E(\bar{R})\}^2}\Big\{\bar{R}\vee c - E(\bar{R})\Big\} + r^*(\bar{V}, \bar{R}),  \label{t-eq-thm1-1}
\end{eqnarray*}
where $r^*(\bar V, \bar R)$ is the remainder term in the Taylor expansion. Since $E(\bar R) > c$ and $\bar R\vee c\geq c$, and in an arbitrary order of the partial derivatives of the function $H(v,r)$, only $r$ can appear in the denominator, therefore, we can verify 
\begin{eqnarray}
|r^*(\bar V, \bar R)| \lesssim \left\{\bar V - E(\bar V)\right\}^2 + \left\{\bar R \vee c - E(\bar R)\right\}^2. \label{t-eq-thm1-2}
\end{eqnarray}
Therefore, we can have
\begin{eqnarray}
\mbox{FDP}(t) = E(\bar V)/E(\bar R) +  m(\bar{V},\bar{R}) + \mathcal{J} +r^*(\bar V, \bar R), \label{t-eq-thm1-3}
\end{eqnarray}
where
\begin{eqnarray*}
m(\bar{V},\bar{R}) &=&  \frac{\bar{V}}{E(\bar{R})} - \frac{E(\bar{V})}{\{E(\bar{R})\}^2} \bar{R} \label{t-eq-thm1-4}\\
\mathcal{J} &=& \frac{\bar V}{\bar R} - \frac{\bar V}{\bar R \vee c}.  \label{t-eq-thm1-5}
\end{eqnarray*}
Define $r(\bar V, \bar R) = r^*(\bar V, \bar R) + \mathcal{J}$. To complete the proof of the theorem, it suffices to show that both $E(\mathcal{J}^2)$ and $E\{r^{*2}(\bar V, \bar R)\}$ can be dominated by $\mbox{Var}\{m(\bar{V},\bar{R})\}$.  We consider $E\{r^{*2}(\bar V, \bar R)\}$ first. Based on \eqref{t-eq-thm1-1}, we have
\begin{eqnarray*}
|r^*(\bar V, \bar R)| &\lesssim& \left\{\bar V - E(\bar V)\right\}^2 + \left\{\bar R \vee c - \bar R\right\}^2 + \left\{ \bar R -  E(\bar R)\right\}^2\\
&=&\left\{\bar V - E(\bar V)\right\}^2 + \left\{ c - \bar R\right\}^2\mathrm{1}(\bar R < c)  + \left\{ \bar R -  E(\bar R)\right\}^2\\
&\leq & \left\{\bar V - E(\bar V)\right\}^2 + c^2\mathrm{1}(\bar R < c)  + \left\{ \bar R -  E(\bar R)\right\}^2, \label{t-eq-thm1-6}
\end{eqnarray*}
which leads to 
\begin{eqnarray}
 E\left\{r^{*2}(\bar V, \bar R)\right\}
&\lesssim& E\left\{\bar V - E(\bar V)\right\}^4 + c^4P(\bar R < c)  + E\left\{ \bar R -  E(\bar R)\right\}^4 \nonumber \\
&\leq & E\left\{\bar V - E(\bar V)\right\}^4   + 2  E\left\{ \bar R -  E(\bar R)\right\}^4, \label{t-eq-thm1-7}
\end{eqnarray}
since based on the definition of $c\in (0, 0.5 E(\bar R))$, we have
\begin{eqnarray*}
P(\bar R<c) \leq P(|\bar R - E(\bar R)|> c) \leq \frac{E\left\{\bar R - E(\bar R)\right\}^4}{c^4}. \label{t-eq-thm1-8}
\end{eqnarray*}
Combining \eqref{t-eq-thm1-7} with Lemma \ref{t-lem-7}, we conclude
\begin{eqnarray}
E\{r^{*2}(\bar V, \bar R)\} = o\left(\mbox{Var}\{m(\bar{V},\bar{R})\}\right). \label{t-eq-thm1-9}
\end{eqnarray}

Next, we consider $E(\mathcal{J}^2)$; in particular
\begin{eqnarray*}
\mathcal{J} = \frac{\bar V}{\bar R} - \frac{\bar V}{\bar R \vee c} = \frac{\bar V}{\bar R}\cdot \frac{\bar R\vee c - \bar R}{\bar R \vee c} = \frac{\bar V}{\bar R} \cdot \frac{c-\bar R}{c}\cdot \mathrm{1}(\bar R<c).
\end{eqnarray*}
By noting $\bar V/\bar R\in (0,1)$ and $(c-\bar R)/c \in (0,1)$ when $\bar R<c$, we have
\begin{eqnarray}
E(\mathcal{J}^2) \leq E\{\mathrm{1}(\bar R<c)\} \leq \frac{E\left\{\bar R - E(\bar R)\right\}^4}{c^4} = o\left(\mbox{Var}\{m(\bar{V},\bar{R}))\right). \label{t-eq-thm1-10}
\end{eqnarray}
Combining \eqref{t-eq-thm1-9}, \eqref{t-eq-thm1-10} with \eqref{t-eq-thm1-3}, we complete the proof of Theorem \ref{Thm-5}.

We proceed to show Corollary \ref{Corrollary-asym-var-unknown}. From (\ref{t-eq-thm1-3}), we have that
\begin{eqnarray*}
\mbox{Var} \left\{ \mbox{FDP}(t) \right\} &=& \mbox{Var} \left\{ m(\bar{V},\bar{R}) + \mathcal{J} + r^*(\bar V, \bar R) \right\} \\
&=& \mbox{Var} \left\{ m(\bar{V},\bar{R}) \right\} + \mbox{Var} \left\{ \mathcal{J} \right\} + \mbox{Var} \left\{r^*(\bar V, \bar R) \right\}\\
&& + 2 \mbox{Cov} \left\{ m(\bar{V},\bar{R}), \mathcal{J} \right\} + 2 \mbox{Cov} \left\{ m(\bar{V},\bar{R}), r^*(\bar V, \bar R) \right\}+ 2 \mbox{Cov} \left\{ \mathcal{J}, r^*(\bar V, \bar R) \right\},
\end{eqnarray*}
which together with \eqref{t-eq-thm1-9}, \eqref{t-eq-thm1-10}, and Cauchy-Schwarz inequality immediately leads to
\begin{equation*}
\lim_{p \to \infty} \frac{\mbox{Var} \left\{ \mbox{FDP}(t) \right\}}{\mbox{Var} \left\{ m(\bar{V},\bar{R}) \right\}} = 1.
\end{equation*}
With straightforward evaluation, we can check that the denominator above is identical to $V_1(t) + V_2(t)$. We complete the proof of Corollary \ref{Corrollary-asym-var-unknown}. 

\bibliographystyle{asa}
\bibliography{FDP_t}